\documentclass[DIV=14,letterpaper]{scrartcl}

\date{}

\usepackage{tikz}
\usetikzlibrary{calc}
\usetikzlibrary{matrix}
\usepackage{amsmath,amssymb,amsfonts,amsthm,subfig}
\usepackage{float}
\usepackage{multirow}
\usepackage{tikz}
\usepackage{graphicx}
\usetikzlibrary{positioning}
\usepackage[pdftex]{hyperref} 

\usepackage{graphicx}
\usepackage{chngcntr}
\counterwithin{table}{section}

\newcommand{\V}[1]{\mbox{\boldmath $ #1 $}}

\newcommand{\tr}[1]{\text{tr} #1}
\newcommand{\mJ}[1]{\mathbb{J} #1}
\newcommand{\M}[1]{\mathbb{M} #1}

\theoremstyle{definition}

\newcommand{\bey}{\begin{eqnarray}}
\newcommand{\eey}{\end{eqnarray}}

\newcommand{\beq}{\begin{equation}}
\newcommand{\eeq}{\end{equation}}
\theoremstyle{plain}
\newtheorem{thm}{\hspace{6mm}Theorem}[section]

\newtheorem{co}{\hspace{6mm}Corollary}[section]
\theoremstyle{definition}

\theoremstyle{remark}
\newtheorem{exam}{\hspace{6mm}Example}[section]

\title{A New Functional for Variational Mesh Generation and Adaptation Based on Equidistribution
and Alignment Conditions}

\author{Avary~Kolasinski%
\thanks{Department of Mathematics, the University of Kansas, Lawrence, KS 66045
({\em avaryk@ku.edu}).}
\and Weizhang~Huang%
\thanks{Department of Mathematics, the University of Kansas, Lawrence, KS 66045
({\em whuang@ku.edu}).}
}

\begin{document}
\vskip 1cm
\maketitle

\begin{abstract}
A new functional is presented for variational mesh generation and adaptation.
It is formulated based on combining the equidistribution and alignment conditions
into a single condition with only one dimensionless parameter.
The functional is shown to be coercive but not convex.
A solution procedure using a discrete moving mesh partial differential
equation is employed. It is shown that the element volumes and altitudes of
a mesh trajectory of the mesh equation associated with the new functional
are bounded away from zero and the mesh trajectory stays nonsingular
if it is so initially. Numerical examples demonstrate that the new functional
performs comparably as an existing one that is also based on the equidistribution
and alignment conditions and known to work well but contains an additional parameter.
\end{abstract}



\noindent{\textbf{Abbreviated title.} A New Functional for Variational Mesh Generation.}

\section{Introduction}

Variational mesh generation and adaptation has proven a useful tool in the numerical solution
of partial differential equations (PDEs); e.g., see \cite{C, HR, KS, L, TWM} and references therein.
In this, a (adaptive) mesh is generated as the image of a reference mesh under a coordinate transformation which is determined as the minimizer of a meshing functional. One of the main advantages of this approach
of mesh generation is that different mesh requirements such as smoothness, orthogonality, adaptivity,
alignment, etc. can easily be incorporated into the formulation of the meshing functional \cite{BS}.
In addition to being a method for mesh generation and adaptation, this approach can also be used
as a smoothing device for automatic mesh generation \cite{FG, HKS} and a base for
adaptive moving mesh methods \cite{HRR, HR99,HR, LTZ}.

There exists a vast literature on variational mesh generation and adaptation. A number of meshing functionals
have been developed from different problems and formulated based on different focused requirments.
For example, Winslow \cite{W} develops an equipotential method that is based on variable diffusion.
Brackbill and Saltzmann \cite{BS} combine mesh concentration, smoothness, and orthogonality to create a functional.
Dvinsky \cite{D} develops a method based on the energy of harmonic mappings.
Knupp \cite{K} and Knupp and Robidoux \cite{KR} focus on the idea of conditioning the Jacobian matrix
of the coordinate transformation.
Huang \cite{H} and Huang and Russell \cite{HR} have proposed two methods based on
the so-called equidistribution and alignment conditions. 

Compared to the algorithmic development, very few theoretical results are known.
For example, Dvinsky's meshing functional \cite{D} is guaranteed to have a unique invertible minimizer by
the theory of harmonic mappings between multidimensional domains. Winslow's functional \cite{W}
is known to have a unique minimizer due to its uniformly convexity and coercivity. Furthermore, the functional by Huang \cite{H} is coercive and polyconvex and thus has minimizers \cite{HR}.
Recently, a new formulation of the so-called moving mesh partial differential equation (MMPDE)
method \cite{HRR, HR99} was proposed by Huang and Kamenski \cite{HK2},
where the meshing functional is first discretized
and then the mesh equation (which will be referred to as the discrete MMPDE hereafter)
is defined as a gradient system of the discretized functional.
This new formulation provides an explicit, compact, and analytical formula for the mesh velocity, which
makes the implementation of the method much easier and more robust (cf. Section~\ref{SEC:mmpde}).
More importantly, several important properties of the discrete MMPDE can be established;
see \cite{HK1} and/or Section~\ref{SEC:theory} for detail. In particular, the mesh trajectory of the discrete MMPDE
stays nonsingular if it is so initially provided that the meshing functional under consideration satisfies
a coercivity condition (cf. (\ref{Coereq}) below).
To our best knowledge,  this is the only nonsingularity result at the discrete level
available in the context of variational mesh generation and adaptation and mesh movement.

It is noted that the functional of \cite{H} satisfies the coercivity condition for a large range
of its parameters. It works well with the framework of MMPDEs and has been successively
used for various applications \cite{HR}. The functional is formulated based on
the equidistribution and alignment conditions -- more precisely, based on an averaging
of the two conditions with a dimensionless parameter.  Although the performance
of the functional does not seem sensitive to the value of the parameter, its choice
is still arbitrary and there is hardly a convincing guideline for choosing it.


The objective of this paper is to present a new functional using the equidistribution and alignment conditions.
Like the existing functional of \cite{H}, this new one is also based on a combination of the two conditions
into a single one, but this time, without introducing any new parameter.
We will show that the new functional satisfies the coercivity condition
and has similar theoretical properties as the existing functional when employed with the MMPDE.
Two-dimensional numerical results will be presented to verify theoretical findings as well as demonstrate
comparable performances of the two functionals.

It is worth pointing out that variational mesh adaptation is a special type of anisotropic mesh adaptation
which has become an area of intensive research. There is a vast literature in this area;
for example, some of the earlier works are \cite{A, AD, BG, BH, CH, DS, FP, HDB, HM, PV, RL, VD, Y}.

An outline of this paper is as follows. In Section~\ref{SEC:functionals}, the equidistribution and alignment conditions
will be introduced and the existing and new functionals will be described.
The discrete MMPDE will be presented as a solution procedure for the minimization problem associated with
a meshing functional in Section~\ref{SEC:mmpde}.  Section~\ref{SEC:theory} is devoted to
the study of the theoretical properties of the new functional, followed by the numerical examples
in Section~\ref{SEC:numerics}. Finally, Section~\ref{SEC:conclusion} contains conclusions and further comments.

\section{Meshing functionals based on equidistribution and alignment conditions}
\label{SEC:functionals}

In this section we are going to describe two meshing functionals that are formulated from the equidistribution and alignment
conditions (cf. ($\ref{equ}$) and ($\ref{ali}$) below).  These conditions have been developed based on the concept
of uniform meshes in some metric tensor \cite{HR}. They provide total control of the mesh element size, shape,
and orientation of mesh elements through a metric tensor. One of the meshing functionals to be described
was first introduced in \cite{H} and involves averaging functionals associated with the two conditions.
It has a number of advantages (which will be discussed later) and is known to work well in practice but
involves two dimensionless parameters. The second functional is new. It is formulated by directly combining
the equidistribution and alignment conditions into a single condition which in turn has eliminated one of
the two parameters of the existing functional.

\subsection{The equidistribution and alignment conditions}

Let the physical domain, $\Omega\subset\mathbb{R}^d$, $d\ge 1$, be a bounded (not necessarily convex) polygonal or polyhedral domain and $\mathbb{M}=\mathbb{M}(\V x)$ be a given symmetric, uniformly positive definite metric tensor defined on $\Omega$ which satisfies
\begin{equation}
\underline{m}I\le \mathbb{M}(\V x)\le \overline{m}I,\quad\forall x\in\Omega,
\label{M-1}
\end{equation}
where $\underline{m}$ and $\overline{m}$ are positive constants and $I$ is the identity matrix.  
Our goal is to generate a simplicial mesh for $\Omega$ which is uniform with respect to the metric $\M$.
Denote this target mesh by $\mathcal{T}_h=\{K\}$ and let $N$ and $N_v$
be the number of its elements and vertices, respectively. 
Assume that the reference element $\hat{K}$ has been chosen to be equilateral and unitary (i.e., $|\hat{K}|=1$,
where $|\hat{K}|$ denotes the volume of $\hat{K}$).
For any element $K\in \mathcal{T}_h$ 
let $F_K:\hat{K}\to K$ be the affine mapping between them and $F_K'$ be its Jacobian matrix. Denote the vertices of $K$ by $\V x_j^K$, $j=0,...,d$ and the vertices of $\hat{K}$ by $\V \xi_j$, $j=0,...,d$. Then 
$$
\V x_j^K=F_K(\V\xi_j).
$$

With this in mind, we can define the equidistribution and alignment conditions
that completely characterize a non-uniform mesh. Indeed, any non-uniform mesh
can be viewed as a uniform one in some metric tensor. Using this viewpoint it is shown (e.g., see \cite{HR}) that a uniform mesh in the metric $\M$ satisfies
\begin{align}
\label{equ}
\text{equidistribution:} & \hspace{0.5cm} |K|\det(\mathbb{M}_K)^{\frac{1}{2}}=\frac{\sigma_h}{N},
~~~\forall K\in\mathcal{T}_h\\ 
\label{ali}
\text{alignment:} &\hspace{0.5cm}  \frac{1}{d}\tr\left((F_K')^{-1}\mathbb{M}_K^{-1}(F_K')^{-T}\right)=\det\left((F_K')^{-1}\mathbb{M}_K^{-1}(F_K')^{-T}\right)^{\frac{1}{d}},~~~\forall K\in\mathcal{T}_h
\end{align}
where $\M_K$ is the average of $\M$ over $K$ and 
\begin{equation}
\label{sigmah}
\sigma_h=\sum_{K\in\mathcal{T}_h}|K|\det(\mathbb{M}_K)^{\frac{1}{2}}.
\end{equation}
Notice that $|K|\det(\mathbb{M}_K)^{\frac{1}{2}}$ is the volume of $K$ in the metric $\mathbb{M}_K$ and thus
the equidistribution condition essentially requires that all of the elements have
the same volume with respect to the metric $\M$.
On the other hand, the left- and right-hand sides of the alignment condition (\ref{ali}) are the arithmetic mean
and geometric mean of the eigenvalues of the matrix $(F_K')^{-1}\mathbb{M}_K^{-1}(F_K')^{-T}$, respectively.
Thus, the condition implies that the eigenvalues of the matrix be equal, i.e., 
\begin{equation}
(F_K')^{-1}\mathbb{M}_K^{-1}(F_K')^{-T} = \theta_K I,
\label{ali2}
\end{equation}
where $\theta_K$ is a positive constant.
It can be shown \cite{HR} that geometrically, the condition (\ref{ali}) requires all elements $K$,
when measured in the metric $\mathbb{M}_K$, to be similar to the reference element $\hat{K}$.
Combining the equidistribution and alignment conditions, we see that if a mesh satisfies
both of them then all of its elements have the same volume and are similar to the reference element,
thus are uniform with respect to the metric $\M$.

\subsection{The existing functional}

We now describe the existing meshing functional based on the equidistribution and alignment conditions.  
First consider the equidistribution condition ($\ref{equ}$). From H\" older's inequality, for any $p>1$ then
\begin{equation}\label{inequ}\left(\sum_{K\in\mathcal{T}_h}\dfrac{|K|\det(\mathbb{M}_K)^{\frac{1}{2}}}{\sigma_h}\cdot\left(\dfrac{1}{|K|\det(\mathbb{M}_K)^{\frac{1}{2}}}\right)^p\right)^{\frac{1}{p}}\ge\sum_{K\in\mathcal{T}_h}\dfrac{|K|\det(\mathbb{M}_K)^{\frac{1}{2}}}{\sigma_h}\cdot\left(\dfrac{1}{|K|\det(\mathbb{M}_K)^{\frac{1}{2}}}\right),\end{equation}
with equality if and only if
$$
\frac{1}{|K|\det(\mathbb{M}_K)^{\frac{1}{2}}}=\text{constant},\quad \forall K\in \mathcal{T}_h.
$$
That is, minimizing the difference between the left-hand side and the right-hand side of ($\ref{inequ}$)
tends to make $1/(|K|\det(\mathbb{M}_K)^{\frac{1}{2}})$ constant for all $K\in\mathcal{T}_h$. 
Noticing that the right-hand side of ($\ref{inequ}$) is ${N}/{\sigma_h}$, we can rewrite this inequality into
\begin{equation}
\label{IN}
\sum_{K\in\mathcal{T}_h}|K|\det(\mathbb{M}_K)^{\frac{1}{2}}
\cdot\left(\dfrac{1}{|K|\det(\mathbb{M}_K)^{\frac{1}{2}}}\right)^p\ge\left(\dfrac{N}{\sigma_h}\right)^p\cdot \sigma_h.\end{equation}
Since $\sigma_h\approx \int_{\Omega}\det(\M)^{\frac{1}{2}}d\V x$, it depends on the mesh only weakly so we can consider $\sigma_h$ to be a constant. Therefore, we can use the left-hand side of ($\ref{IN}$) as the functional for the equidistribution condition.
Noticing that $\det(F_K')=|K|$ we thus have
\begin{equation}
\label{Ieq}
I_{eq}(\mathcal{T}_h)=d^{\frac{dp}{2}}\sum_{K\in\mathcal{T}_h} |K|\det(\mathbb{M}_K)^{\frac{1}{2}}\left(\det(F_K')^{-1}\det(\mathbb{M}_K)^{-\frac{1}{2}}\right)^p.
\end{equation}

We now consider the alignment condition ($\ref{ali}$). Recall that its left- and right-hand sides are
the arithmetic and geometric mean of the eigenvalues of the matrix $(F_K')^{-1}\M_K^{-1}(F_K')^{-T}$, respectively.
By the arithmetic-mean geometric-mean inequality, we have 
\begin{equation}
\label{IN2}
\frac{1}{d}\tr\left((F_K')^{-1}\M_K^{-1}(F_K')^{-T}\right)\ge \det\left((F_K')^{-1}\M_K^{-1}(F_K')^{-T}\right)^{\frac{1}{d}},\end{equation}
with equality if and only if all of the eigenvalues are equal. From this, we have
\[
\left(\tr\left((F_K')^{-1}\M_K^{-1}(F_K')^{-T}\right)\right)^{\frac{dp}{2}}\ge d^{\frac{dp}{2}} \left(\det(F_K')^{-1}
\det(\M_K)^{-\frac12}\right)^p
\]
and
\begin{align*}
& \sum_{K\in\mathcal{T}_h}|K|\det(\M_K)^{\frac{1}{2}}\left(\tr\left((F_K')^{-1}\M_K^{-1}(F_K')^{-T}\right)\right)^{\frac{dp}{2}}
\\
& \qquad \qquad \ge \sum_{K\in\mathcal{T}_h}|K|\det(\M_K)^{\frac{1}{2}}d^{\frac{dp}{2}}
\left(\det(F_K')^{-1} \det(\M_K)^{-\frac 1 2}\right)^p,
\end{align*}
where $p>0$. Minimizing the difference of the left- and right-hand sides makes the mesh tend to satisfying the alignment condition. Therefore, we can define our alignment functional as
\begin{equation}
\label{Iali}
 I_{ali}(\mathcal{T}_h)=\sum_{K\in\mathcal{T}_h}|K|\det(\mathbb{M}_K)^{\frac{1}{2}}
 \left[\tr\left((F_K')^{-1}\mathbb{M}_K^{-1}(F_K')^{-T}\right)^{\frac{dp}{2}}
 -d^{\frac{dp}{2}}\left(\frac{1}{\det(F_K')\det(\mathbb{M}_K)^{\frac{1}{2}}}\right)^p\right].
\end{equation}

We now have two functionals and want to obtain a mesh that tries to minimize both.
One way to ensure this is to combine the two functionals into a single one.
For example, we can average the equidistribution functional ($\ref{Ieq}$)
and the alignment functional ($\ref{Iali}$) with a dimensionless parameter $\theta\in (0,1)$, i.e.,
\begin{align}
\nonumber I_h(\mathcal{T}_h)= &~\theta I_{ali}(\mathcal{T}_h)+(1-\theta)I_{eq}(\mathcal{T}_h)\\
=& \nonumber~\theta\sum_{K\in\mathcal{T}_h}|K|\det(\mathbb{M}_K)^{\frac{1}{2}}
\left(\tr\left((F_K')^{-1}\mathbb{M}_K^{-1}(F_K')^{-T}\right)\right)^{\frac{dp}{2}}\\
\label{Huang1}
&\qquad +(1-2\theta)d^{\frac{dp}{2}}\sum_{K\in\mathcal{T}_h}|K|\det(\mathbb{M}_K)^{\frac{1}{2}}\left(\det(F_K')^{-1}\det(\mathbb{M}_K)^{-\frac 12}\right)^p.
\end{align}
This functional was first proposed in \cite{H} in the continuous form.  As one can notice, the equidistribution
and alignment conditions are balanced in equation ($\ref{Huang1}$) by the dimensionless parameter $\theta$,
for which full alignment is achieved when $\theta=1$ and full equidistribution is achieved when $\theta=0$.
For $0<\theta\le\frac{1}{2}$, $dp\ge 2$, and $p\ge 1$, the functional is coercive and polyconvex
and thus has a minimizer \cite{HR}. It has been shown in \cite{HK} that the MMPDE mesh equation (cf. Sect. 3)
associated with this functional has a mesh trajectory that stays nonsingular for all time
and has element volumes and altitudes bounded away from zero.
The functional has also been successfully used for many problems.
 
\subsection{The new functional}
The existing functional contains two parameters which have large disadvantages, and particularly it is still unclear how to choose an optimal $\theta$. Ideally we would like to take $\theta=1/2$ to ensure (\ref{Huang1}) is convex, but, unfortunately, previous numerical experiments show that this choice of $\theta$ does not put enough emphasis
on the equidistribution condition which controls the mesh concentration.
Moreover, larger values of $\theta$ emphasize the alignment condition which produces a more regular mesh. However, this regularity can also be achieved by choosing larger values of $p$ \cite{HR}.
This relation between $\theta$ and $p$ is not very clear. It has been known experimentally
that $\theta=1/3$ and $p = 3/2$ work well for many problems.
Here, we consider a new functional that eliminates the additional parameter $\theta$.
To this end, we first notice that ($\ref{equ}$) and ($\ref{ali}$) can be cast in a single condition.
Indeed, taking the determinant of both sides of ($\ref{ali2}$), we get
$$
\theta_K^d=\text{det}((F_K')^{-T}\mathbb{M}_K^{-1} F_K'^{-1})
=\text{det}(F_K')^{-2}\text{det}(\mathbb{M}_K)^{-1}=|K|^{-2}\det(\mathbb{M}_K)^{-1},
$$
which gives
$$
|K|\det(\mathbb{M}_K)^{\frac{1}{2}}=\theta_K^{-\frac d 2}.
$$
Comparing this to the equidistribution condition ($\ref{equ}$) we get 
$$
\theta_K=\left(\frac{\sigma_h}{N}\right)^{-\frac 2 d}.
$$
Thus, we obtain a single condition
$$
\left(F_K'\right)^{-T}\mathbb{M}_K^{-1}\left(F_K'\right)^{-1}=\left(\frac{\sigma_h}{N}\right)^{-\frac 2 d}I,
~~~\forall K\in \mathcal{T}_h
$$
which directly combines the equidistribution and alignment conditions. From this, we can define a new functional as
\begin{equation} \label{Huang2}
I_h=\sum_{K\in\mathcal{T}_h}|K|\det(\mathbb{M}_K)^{\frac{1}{2}}\left\|(F_K')^{-1}\mathbb{M}_K^{-1}(F_K')^{-T}-\left(\frac{\sigma_h}{N}\right)^{-\frac 2 d}I\right\|_{F}^{2p},
\end{equation}
where $\sigma_h$ is given in ($\ref{sigmah}$) and $\|\cdot\|_F$ is the Frobenius norm for matrices.  Generally speaking, since we are working with $d\times d$ matrices, we can use any matrix norm and produce an equivalent form of the functional.  We choose the Frobenius norm because it is convenient to compute.  We remark that the weight, $|K|\det(\mathbb{M}_K)^{\frac{1}{2}}$, is chosen so that ($\ref{Huang2}$) is more comparable to ($\ref{Huang1}$) which includes the energy functional of a harmonic mapping as a special example. Furthermore, this weight factor is used to emphasize the region where $\det(\M)$ (error density) is large.

Minimizing ($\ref{Huang2}$) will then ensure that the mesh satisfies both the equidistribution and alignment conditions
as closely as possible.



Notice that this functional only contains one parameter, $p$.  In Sect.~\ref{SEC:theory},
it will be proven that this new functional has similar theoretical properties as the existing functional.

\section{The moving mesh PDE solution strategy}
\label{SEC:mmpde}

In principle, we can directly minimize the two functionals ($\ref{Huang1}$) and ($\ref{Huang2}$)
given in the last section, however, this direct minimization problem is too difficult due to their extreme nonlinearity.
Instead, we will employ the moving mesh PDE (MMPDE) method \cite{HR} to find the minimizer.
To be specific, we define the mesh equation as a modified gradient system of $I_h$, i.e.,
\begin{equation}\label{MMPDEx-0}
\dfrac{d\V x_i}{dt}=-\dfrac{P_i}{\tau}\left(\dfrac{\partial I_h}{\partial \V x_i}\right)^T,~~~i=1,\dots, N_v
\end{equation}
where ${\partial I_h}/{\partial \V x_i}$ is considered as a row vector, $P_i$ is a positive scalar function used
to make the equation have invariance properties, and $\tau>0$ is a constant parameter used to adjust the time scale
of mesh movement.  It is interesting to notice that integrating ($\ref{MMPDEx-0}$) is equivalent to solving the minimization problem using the fastest descent method. The analytical formulation of the gradient
${\partial I_h}/{\partial \V x_i}$ has been obtained by Huang and Kamenski \cite{HK2}
for functionals in a general form
\[
I_h=\sum_{K\in\mathcal{T}_h}|K|G\left(\left(F_K'\right)^{-1},\det\left(F_K'\right)^{-1}, \M_K\right),
\]
where $G=G(\mJ,\det(\mJ),\M)$ is a smooth function of three arguments.
Using the formulation, we can rewrite the mesh equation in a compact form as
\begin{equation}
\dfrac{d \V x_i}{dt}=\dfrac{P_i}{\tau}\sum_{K\in \omega_i}|K|\V v_{i_K}^K,\quad i = 1, ..., N_v
\label{MMPDEx}
\end{equation} 
where $\omega_i$ is the patch of elements having $\V x_i$ as one of their vertices and $i_K$ and $\V v_{i}^K$ are the local index and velocity of $\V x_i$ on $K$, respectively. The local velocities are given by
\begin{align*}
\begin{bmatrix}
         (\V v_1^K)^T \\
         \vdots \\
         (\V v_d^K)^T
        \end{bmatrix}
&= -G E_K^{-1}+E_K^{-1}\frac{\partial G}{\partial \mathbb{J}} \hat{E} E_K^{-1}+\frac{\partial G}{\partial \det(\mathbb{J})}\frac{\det(\hat{E})}{\det(E_K)}E_K^{-1}\\
& \qquad -\frac{1}{d+1}\sum_{j=0}^d\text{tr}\left(\frac{\partial G}{\partial \mathbb{M}}\mathbb{M}_{j,K}\right)\begin{bmatrix}
         \frac{\partial \phi_{j,K}}{\partial \V x} \\
         \vdots \\
          \frac{\partial \phi_{j,K}}{\partial \V x}
        \end{bmatrix} ,
\end{align*}
$$
(\V v_0^K)^T=-\sum_{k=1}^d(\V v_k^K)^T-\sum_{j=0}^d\text{tr}\left(\frac{\partial G}{\partial \mathbb{M}}\mathbb{M}_{j,K}\right)\frac{\partial \phi_{j,K}}{\partial \V x},
$$
where $\mathbb{M}_{j,K}=\mathbb{M}(\V x_j^K)$, $\phi_{j,K}$ is a linear basis function associated
with $\V x_j^K$, $\frac{\partial \phi_{j, K}}{\partial \V x}$ is the gradient of $\phi_{j,K}$ as a row vector,
and $E_K$ and $\hat{E}$ are the edge matrices defined as
\[
E_K=[\V x_1^K-\V x_0^K,....,\V x_d^K-\V x_0^K], \quad \hat{E}=[\V \xi_1-\V \xi_0,....,\V \xi_d-\V \xi_0].
\]
Thus, in order to calculate the above velocities, we need 
$$
G, \hspace{.4cm} \frac{\partial G}{\partial \mathbb{J}}, \hspace{.4cm} \frac{\partial G}{\partial \det(\mathbb{J})}, \hspace{.4cm} \frac{\partial G}{\partial \mathbb{M}},
$$
where the derivatives are scalar-by-matrix derivatives as shown in \cite{HK2} and
$$
\mathbb{J}=(F_K')^{-1}=\hat{E}E^{-1}_K, \quad~~ \det(\mathbb{J})=\det(F_K')^{-1}=\frac{\det(\hat{E})}{\det(E_K)}, \quad~~ \V{x}=\V{x}_K, \quad~~ \M=\M_K .
$$

For the existing functional ($\ref{Huang1}$), we have
\begin{align}
\nonumber G\left(\mJ , \det(\mJ ),\mathbb{M}\right)
=& ~\theta\det(\mathbb{M})^{\frac{1}{2}}\left(\tr(\mJ\mathbb{M}^{-1}\mJ^{T})\right)^{\frac{dp}{2}}
+(1-2\theta)d^{\frac{dp}{2}}\det(\mathbb{M})^{\frac{1}{2}}\left(\det(\mJ)\det(\mathbb{M})^{-\frac12}\right)^p.
\label{Huang1G} 
\end{align}
The derivatives of $G$ in this case are given by
\[
\begin{cases}
        \dfrac{\partial G}{\partial \mathbb{J} }
        &=\; dp\theta\sqrt{\det(\mathbb{M})}\left(\tr(\mJ\mathbb{M}^{-1}\mJ^{T})\right)^{\frac{dp}{2}-1}\mathbb{M}^{-1}\mJ^{T}, \\
      \dfrac{\partial G}{\partial \det(\mJ)}&=\;  p(1-2\theta)d^{\frac{dp}{2}}\det(\mathbb{M})^{\frac{1-p}{2}}\det(\mJ)^{p-1},\\
      \dfrac{\partial G}{\partial \mathbb{M}}&=\; -\frac{\theta dp}{2}\sqrt{\det(\mathbb{M})}\left(\tr(\mJ\mathbb{M}^{-1}\mJ^{T})\right)^{\frac{dp}{2}-1}\mathbb{M}^{-1}\mJ^{T}\mJ\mathbb{M}^{-1}\\
      &\qquad +\; \frac{\theta}{2}\sqrt{\det(\mathbb{M})}\left(\tr(\mJ\mathbb{M}^{-1}\mJ^{T})\right)^{\frac{dp}{2}}\mathbb{M}^{-1}\\
     &\qquad +\; \frac{(1-2\theta)(1-p)d^{\frac{dp}{2}}}{2}\sqrt{\det(\mathbb{M})}
     \left(\dfrac{\det(\mJ)}{\sqrt{\det(\mathbb{M}}}\right)^p\mathbb{M}^{-1}.
\end{cases}
\]
For the new functional ($\ref{Huang2}$), we have
\beq \label{Huang2G}
G\left(\mJ,\det(\mJ),\mathbb{M}\right)=\sqrt{\det(\mathbb{M})}\left\|\mJ\mathbb{M}^{-1}\mJ^{T}-\left(\dfrac{\sigma_h}{N}\right)^{-\frac 2d}I\right\|_{F}^{2p}.
\eeq
The derivatives of $G$ for this functional are  
\[   
\begin{cases}
      \dfrac{\partial G}{\partial \mathbb{J}}
      &=\; 4 p \left\|\mJ\mathbb{M}^{-1}\mJ^{T} -\left(\dfrac{\sigma_h}{N}\right)^{-\frac 2d}I\right\|_{F}^{2(p-1)} \sqrt{\det(\mathbb{M})}~\mathbb{M}^{-1}\mJ^{T}\left(\mJ\mathbb{M}^{-1}\mJ^{T}
      -\left(\dfrac{\sigma_h}{N}\right)^{-\frac 2d}I\right), \\
            \dfrac{\partial G}{\partial \det(\mJ)}& = \; 0,\\
      \dfrac{\partial G}{\partial \mathbb{M}}&=\; \dfrac{1}{2}G\M^{-1}-\dfrac{1}{2}\dfrac{\partial G}{\partial \mathbb{J}}\mJ\M^{-1}.
\end{cases}
\]
Note that in the above derivation, we have viewed $\sigma_h$ as a constant since $\sigma_h\sim \int_{\Omega}\det(\M)^{\frac{1}{2}}d\V x.$

It is remarked that the mesh equation (\ref{MMPDEx}) needs to be modified for boundary vertices. For example,
we need to set the velocity to zero for corner vertices. For other boundary vertices, the velocity
should be modified so that they only slide along the boundary.
With appropriate modifications for boundary vertices and for a given metric tensor $\M$, (\ref{MMPDEx})
can be integrated for an adaptive mesh. We use Matlab's {\em ode15s} (a variable-order ODE solver
based on the numerical differentiation formulas) in our computation.

\section{Theoretical analysis of the new functional}
\label{SEC:theory}

In this section we study properties of the new functional ($\ref{Huang2}$).
In particular, we are interested in the coercivity, which is known to be
key to showing the nonsingularity and convergence of the mesh trajectory \cite{HK1}. 
We also study the non-singularity of the mesh trajectory and prove the existence
of limit meshes as $t\to \infty$ for the semi-discrete MMPDE ($\ref{MMPDEx-0}$).

\subsection{Coercivity}
\label{SEC:Coercivity}

\begin{thm}
The new functional ($\ref{Huang2}$) with $p > 1$ is coercive, 
i.e., there exist positive constants $\alpha$ and $\beta$ such that the function $G$ defined in (\ref{Huang2G}) satisfies
\beq
\label{Coereq}
G\ge\alpha\left\|\mJ\right\|_{F}^{4p}-\beta.
\eeq
\label{coercivity}
\end{thm}

\begin{proof}
For notational simplicity, we denote $\gamma_h=\left(\frac{\sigma_h}{N}\right)^{-2/d}$. From the triangle inequality and H\"older's inequality, we have
\begin{align*}
\left\| \mJ\mathbb{M}^{-1}\mJ^{T}-\gamma_hI\right\|_F^{2p}&\ge\left(\left\| \mJ\mathbb{M}^{-1}\mJ^{T}\right\|_F-\|\gamma_hI\|_F\right)^{2p}\\
&\ge2^{1-2p}\left\| \mJ\mathbb{M}^{-1}\mJ^{T}\right\|_F^{2p}-\gamma_h^{2p}\|I\|_F^{2p}\\
&=2^{1-2p}\left\|\mJ\mathbb{M}^{-1}\mJ^{T}\right\|_F^{2p}-\left(\gamma_h^{2}d\right)^p.
\end{align*}
Notice that for a $d\times d$ matrix $A$, we know that $\|A\|_2\le \|A\|_F\le \sqrt{d}\|A\|_2$. With this, it follows
$$\|\mJ\M^{-1}\mJ\|_F\ge \|\mJ\M^{-1}\mJ\|_2\ge \frac{1}{\overline{m}}\|\mJ\mJ^T\|_2=\frac{1}{\overline{m}}\|\mJ\|_2^2\ge \frac{1}{\overline{m}d}\|\mJ\|_F^2.
$$
Combining the above results, we get
$$G\ge \underline{m}^{\frac{d}{2}}\|\mJ\M^{-1}\mJ-\gamma_hI\|_F^{2p}\ge\dfrac{2^{1-2p}\underline{m}^{\frac{d}{2}}}{\overline{m}^{2p}d^{2p}}\|\mJ\|_F^{4p}-\underline{m}^{\frac{d}{2}}(\gamma_h^2d)^p.$$
Thus, $G$ satisfies ($\ref{Coereq}$) with $\alpha=\frac{2^{1-2p}\underline{m}^{\frac{d}{2}}}{\overline{m}^{2p}d^{2p}}$ and $\beta=\underline{m}^{\frac{d}{2}}(\gamma_h^2d)^p$.
\end{proof}

Thus the new functional is coercive. Unfortunately, it is not convex. As a consequence, there is no guarantee
that the minimizer of $I_h$ is unique. It does, however, have other important properties that are discussed
in detail next.

\subsection{Nonsingularity of the mesh trajectory}

Consider the semi-discrete MMPDE (\ref{MMPDEx-0}) with the new functional (\ref{Huang2}).
For a given metric tensor $\M$, which is independent of $t$ and satisfies (\ref{M-1}), the MMPDE
will generate a mesh trajectory $\mathcal{T}_h(t), \, t > 0$ for any given nonsingular initial mesh.
We denote the minimum altitude of $K$ in the metric $\mathbb{M}_K$ by $a_{K,\mathbb{M}}$.

\begin{co}
\label{boundcor} 
For any $t > 0$, the elements of the mesh trajectory of the semi-discrete MMPDE (\ref{MMPDEx-0}) with
the new functional ($\ref{Huang2}$) satisfy
\begin{equation} \label{akm}
a_{K,\mathbb{M}}\ge C_1 \overline{m}^{-\frac{d}{2(4p-d)}}N^{-\frac{4p}{d(4p-d)}}, ~~~ \forall K\in\mathcal{T}_h(t) ,
\end{equation}
\begin{equation} \label{K}
|K|\ge C_2 \overline{m}^{-\frac{d^2}{2(4p-d)}-\frac{d}{2}}N^{-\frac{4p}{(4p-d)}},  ~~~ \forall K\in\mathcal{T}_h(t) ,
\end{equation}
where $C_1$ and $C_2$ are constants give by 
\begin{equation}
C_1=\left(\dfrac{2^{6p}~d!^{\frac{4p}{d}}~\alpha}{d^{4p}(d+1)^{4p-\frac{2p}{d}}\left(\beta|\Omega|+I_h\left(\mathcal{T}_h(0)\right)\right)}\right)^{\frac{1}{4p-d}},\quad C_2=\dfrac{C_1^d}{d!},
\end{equation}
and $\alpha$ and $\beta$ are defined in the proof of Theorem \ref{coercivity}.
Moreover, $\mathcal{T}_h(t)$ is nonsingular for all $t>0$ if it is nonsingular initially.
\end{co}

\begin{proof}
This is a consequence of Theorem 4.1 in \cite{HK1} which is stated for a general coercive functional.
A direct application of this theorem with $q=2p$ and Theorem~\ref{coercivity} in the previous subsection 
gives the desired result.
\end{proof}

The key components in the proof of Theorem 4.1 in \cite{HK1} are the coercivity of the functional and
the decreasing energy along the mesh trajectory of (\ref{MMPDEx-0}). The latter can be seen from
\[
\frac{d I_h}{d t} = \sum_{i} \frac{\partial I_h}{\partial \V{x}_i} \frac{d \V{x}_i}{d t}
= - \sum_{i} \frac{P_i}{\tau} \frac{\partial I_h}{\partial \V{x}_i} \left (\frac{\partial I_h}{\partial \V{x}_i}\right )^T
= - \sum_{i} \frac{P_i}{\tau} \left \| \frac{\partial I_h}{\partial \V{x}_i} \right \|^2 \le 0.
\]

The role of the parameter $p$ can be explained to some extent from the inequality (\ref{akm}).
Indeed, from (\ref{akm}) we have
\[
a_{K,\mathbb{M}}\ge C_1 \overline{m}^{-\frac{d}{2(4p-d)}}N^{-\frac{4p}{d(4p-d)}}\to C_1 N^{-\frac{1}{d}},
\quad p\to\infty.
\]
Thus, the mesh becomes more uniform as $p$ is getting larger.

One may notice that the bounds in (\ref{akm}) and (\ref{K}) depend on $N$ and $\overline{m}$. 
This is natural since the elements becomes smaller for larger $N$.
Moreover, from the equidistribution condition ($\ref{equ}$), we can see that $|K| \sim \det(\M_K)^{-\frac 1 2}$, thus we can expect the lower bounds for the altitudes and volumes of the elements to become
smaller as $\overline{m}$ gets larger.

Consider now the fully discrete case. Let $t_n$, $n=0,1,\dots$ denote the time levels
with $t_n\to\infty$ as $n\to \infty$. Assume that we have chosen a one-step integration scheme
for ($\ref{MMPDEx}$) such that the energy is decreasing, i.e.,
\begin{equation}
I_h(\mathcal{T}_h^{n+1})\le I_h(\mathcal{T}_h^n).
\label{energy-2}
\end{equation}
Many schemes such as Euler's and the backward Euler have this property with a sufficiently small
but not diminishing time step; e.g., see \cite{HL,HK2}. Then, Corollary~\ref{boundcor} will also holds
for the mesh sequence, $\mathcal{T}_h^n$, $n = 0, 1, ...$.

\subsection{Limits of the mesh trajectory}

A direct application of Theorem 4.3 in \cite{HK1}, which is stated for a general coercive functional
and Theorem~\ref{coercivity} in Section~\ref{SEC:Coercivity}, gives the following corollary.

\begin{co}
\label{co2} 
The mesh trajectory of the semi-discrete MMPDE (\ref{MMPDEx-0}) with
the new functional (\ref{Huang2}) has the following properties.
\begin{itemize}
\item[(a)] $I_h(\mathcal{T}_h(t))$ has a limit as $t\to \infty$, i.e.,
$$\lim_{t\to\infty}I_h(\mathcal{T}_h)=L.$$
\item[(b)] The mesh trajectory has limit meshes, all of which are non-singular and satisfy the bounds given
in Corollary~\ref{boundcor}.
\item[(c)] The limit meshes are critical points of $I_h$.
\end{itemize}
\end{co}

The result in Corollary~\ref{co2} ensures that as time increases, the values of the functional for the mesh trajectory
converge. This is significantly beneficial since it can be used as a computational stopping criteria.
It should be noted that in general, there is no guarantee the mesh trajectory converges. In order to guarantee this convergence, stronger requirements need to be placed on either the descent in the functional value or on the meshing functional; e.g., see the more detailed discussion in \cite{HK2}. Moreover, like Corollary~\ref{boundcor},
Corollary~\ref{co2} also holds for the fully discrete case provided that the time step is sufficiently small and
the scheme satisfies the energy decreasing condition (\ref{energy-2}).

To conclude this section, we note that the existing functional (\ref{Huang1}) is also coercive for
$p > 1$ and $\theta \in (0, 1/2]$. Thus, Corollaries~\ref{boundcor} and \ref{co2} apply to
the existing functional as well.

\section{Numerical examples}
\label{SEC:numerics}

Here we present numerical results for two examples in two dimensions
to demonstrate the theoretical findings discussed in Section~\ref{SEC:theory}.
Two of the main focuses will be showing the positive lower bound of the element volumes
and the monotonically decreasing energy functional. Additionally, we will provide and compare
meshes associated with the new and existing functionals.
In order to asses the quality of the generated meshes, we compare the linear interpolation error
($error$, measured in the $L^2$ norm), and the equidistribution ($Q_{eq}$),
alignment ($Q_{ali}$), and geometric ($Q_{geo}$) mesh quality measures which are defined as
\begin{equation}
Q_{eq}=\sqrt{\frac{1}{N}\sum_{K\in\mathcal{T}_c} Q_{eq,K}^2},
\quad Q_{ali}=\sqrt{\frac{1}{N}\sum_{K\in\mathcal{T}_h}Q_{ali,K}^2},
\quad Q_{geo}=\sqrt{\frac{1}{N}\sum_{K\in\mathcal{T}_h}Q_{geo,K}^2},
\end{equation}
where 
\begin{equation}
Q_{eq,K}=\dfrac{|K| \det(\mathbb{M}_K)^{\frac{1}{2}}}{\sigma_h/N},
\quad Q_{ali,K}=\dfrac{\tr\left((F_K')^T\mathbb{M}_K F_K'\right)}
{d\det\left((F_K')^T\mathbb{M}_K F_K'\right)^{\frac{1}{d}}},
\quad Q_{geo,K}=\dfrac{\tr\left((F_K')^T F_K'\right)}
{d\det\left((F_K')^T F_K'\right)^{\frac{1}{d}}} .
\end{equation}
The equidistribution and alignment measures are indications of how closely the mesh 
satisfies the equidistribution condition (\ref{equ}) and the alignment condition (\ref{ali}), respectively.
The closer these quality measures are to 1, the closer they are to a uniform mesh with respect to
the metric $\mathbb{M}$. The geometric measure is the same as the alignment quality measure
taking $\mathbb{M}=I$. It measures how skew the mesh is in the Euclidean metric.

We use $p=3/2$ and $\theta=1/3$ in the existing functional (\ref{Huang1}) and $p = 1$ in the new functional (\ref{Huang2}).  The defined parameters $p$ and $\theta$ for the existing functional are commonly used and known to work well for most problems. The choice for $p$ in the new functional is based on the desire to ensure that ($\ref{Huang2}$) is a quadratic function of matrix entries,  which, computationally, makes the MMPDE less difficult to solve.
The parameter $\tau$ in the MMPDE ($\ref{MMPDEx}$) is taken to be $\tau=10^{-2}$.
Additionally, for the positive function $P_i$ in ($\ref{MMPDEx}$) we use $P_i=\det(\M)^{\frac{p-1}{2}}$ for the existing functional and $P_i=\det(\M)^{\frac{2}{d}}$ for the new functional to ensure, for both cases,
that the MMPDE ($\ref{MMPDEx}$) is invariant under the scaling transformation of $\M$.
The two dimensional meshes for Example $\ref{Exam5.1}$ and Example $\ref{Exam5.2}$ are constructed on the domain $\Omega=(0,1)\times(0,1)$.  
We take the metric tensor as
\[
\mathbb{M}_K=\det(|H_K|)^{\frac{-1}{d+4}}|H_K|,
\]
where $H_K$ is the recovered Hessian using least squares fitting to the function values at the mesh vertices and $|H_K| = Q\text{diag}(|\lambda_1|, ..., |\lambda_d|)Q^T$, assuming that
$Q\text{diag}(\lambda_1, ..., \lambda_d)Q^T$ is the eigen-decomposition of $H_K$.
It is known \cite{HR} that the above form of the metric tensor is optimal corresponding to the $L^2$-norm
of linear interpolation on triangular meshes.

\begin{exam}
\label{Exam5.1}
In this example, we generate adaptive meshes for the sine wave modeled by
$$
u(x,y) = \tanh\left( -30 \left[y - 0.5 - 0.25\sin(2\pi x) \right] \right) .
$$
For the following results, we run to a final time of $5.0$.

The example meshes and close-ups are given in Fig.~\ref{fig:exam5.1-mesh}.  The mesh associated
with the new functional provides good shape and size adaptation. There is a high concentration of mesh elements
in regions with large curvature near the interface. This is consistent with the fact that the used metric tensor
is Hessian based. A closer look at the mesh shows that the elements are more skew (in the Euclidean metric)
in the places with larger curvature. This is also shown in Table~\ref{table-5.1} with $Q_{geo} \approx 2$.
On the other hand, $Q_{ali}$ is close to $1$, indicating that the mesh almost satisfies the alignment condition
under the metric $\M$. Therefore, the mesh may seem skew in the Euclidean metric but is very regular
in the metric $\M$.

\begin{figure}[htb]
\begin{center}
\hbox{
\begin{minipage}[t]{2.1in}
\begin{center}
\includegraphics[width=2.1in]{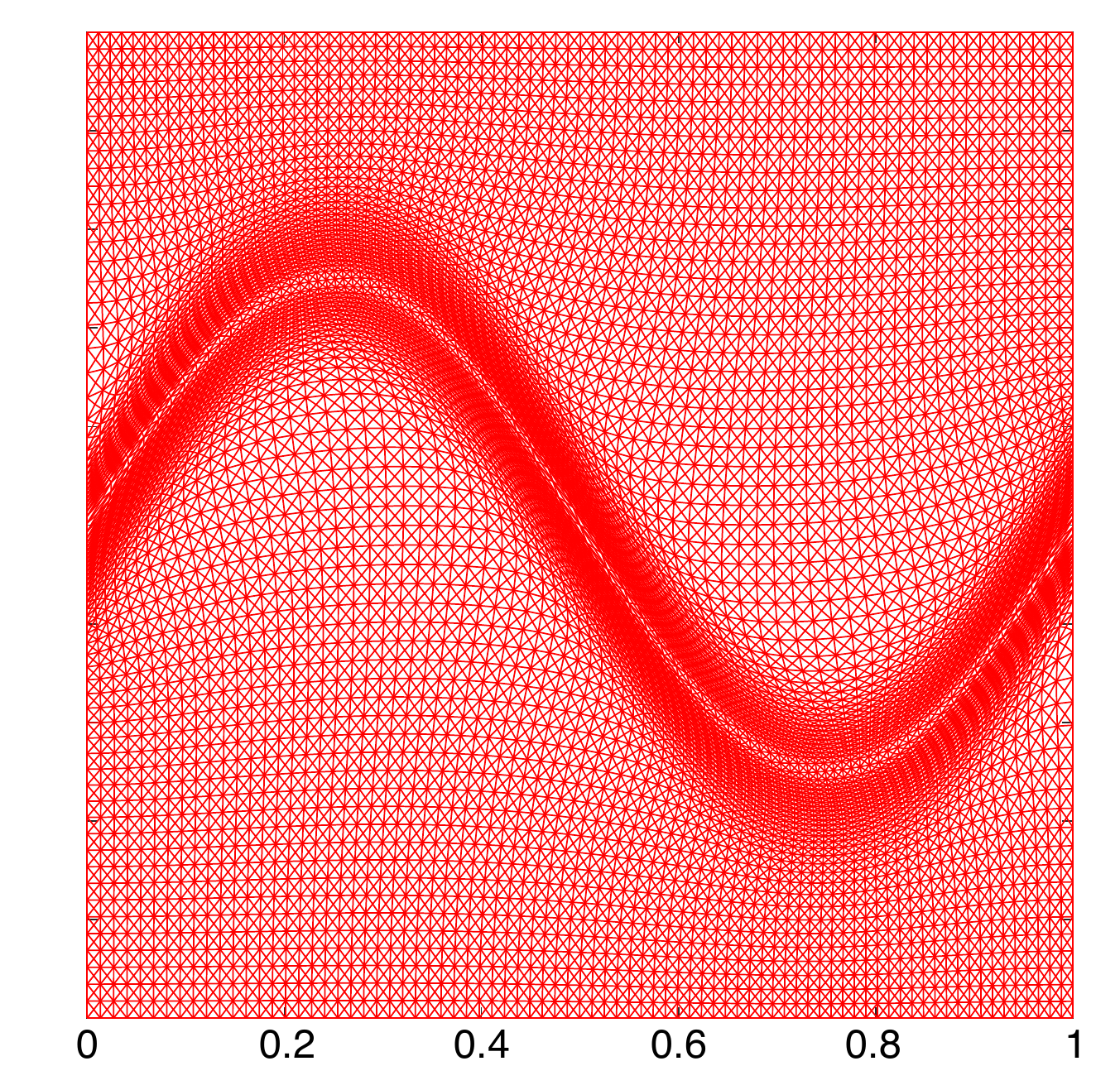}
\end{center}
\centerline{(a) New functional}\
\end{minipage}
\hspace{2mm}
\begin{minipage}[t]{2.1in}
\begin{center}
\includegraphics[width=2.1in]{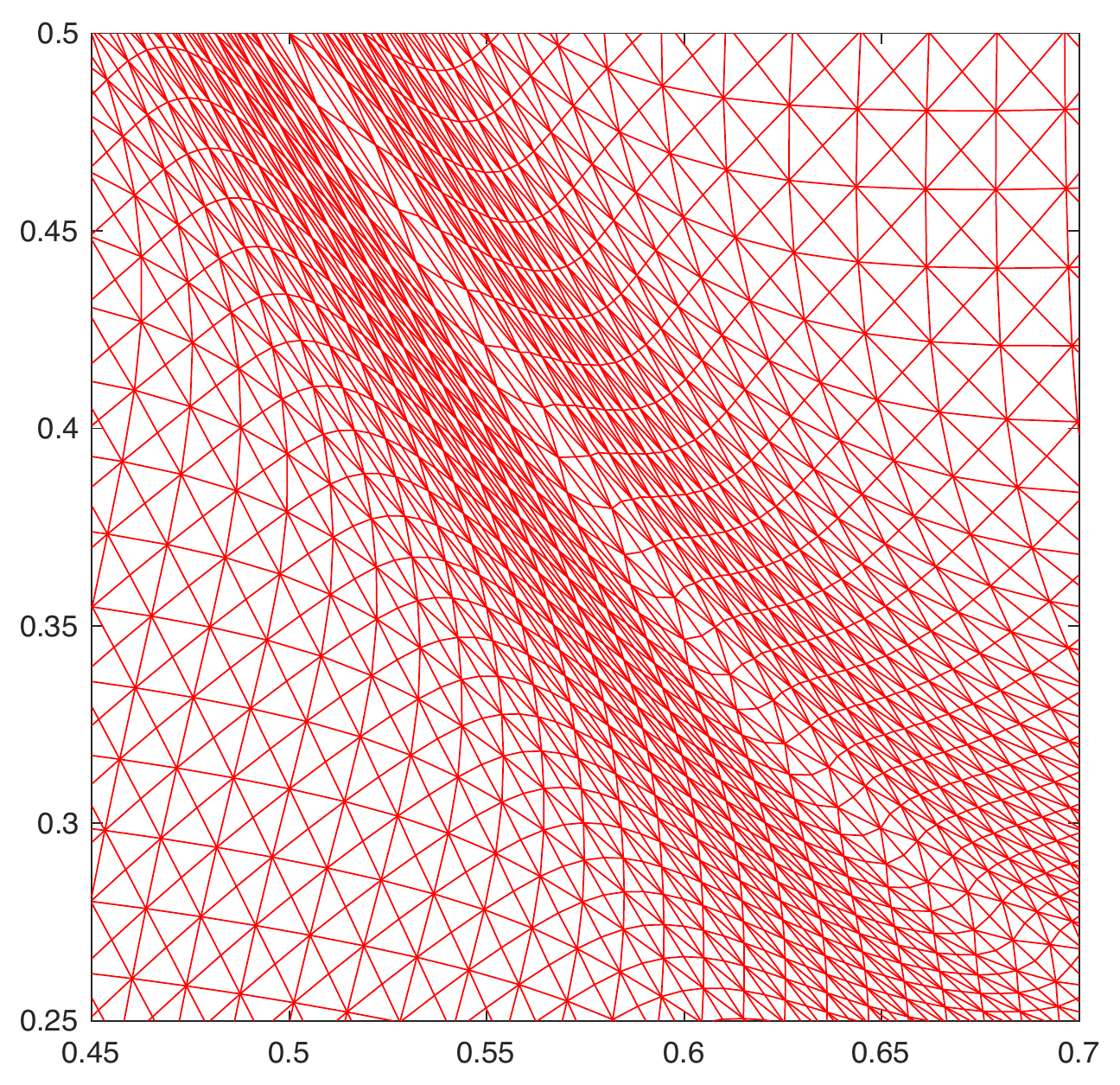}
\end{center}
\end{minipage}
\hspace{2mm}
\begin{minipage}[t]{2.1in}
\begin{center}
\includegraphics[width=2.1in]{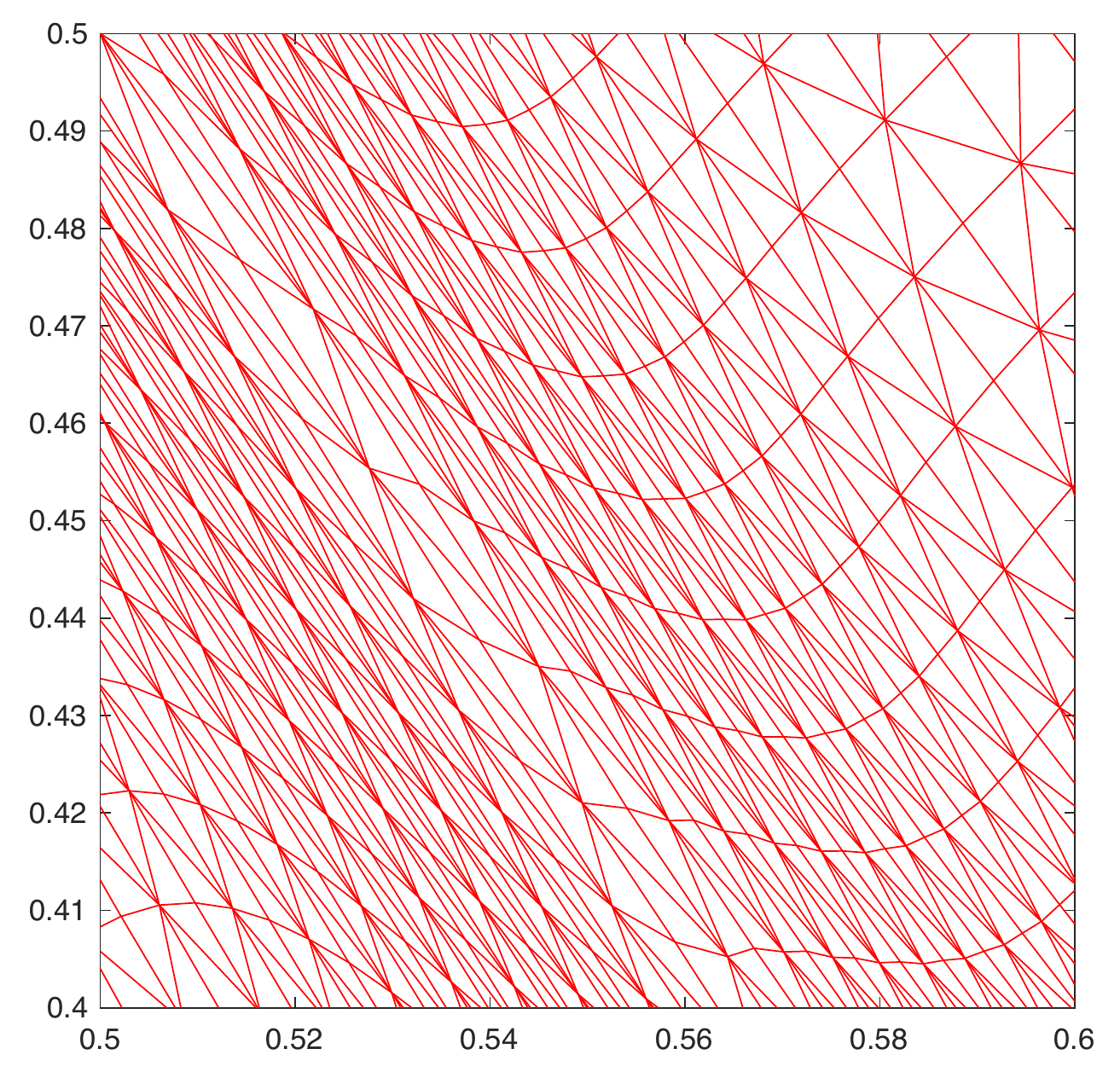}
\end{center}
\end{minipage}
}
\hbox{
\begin{minipage}[t]{2.1in}
\begin{center}
\includegraphics[width=2.1in]{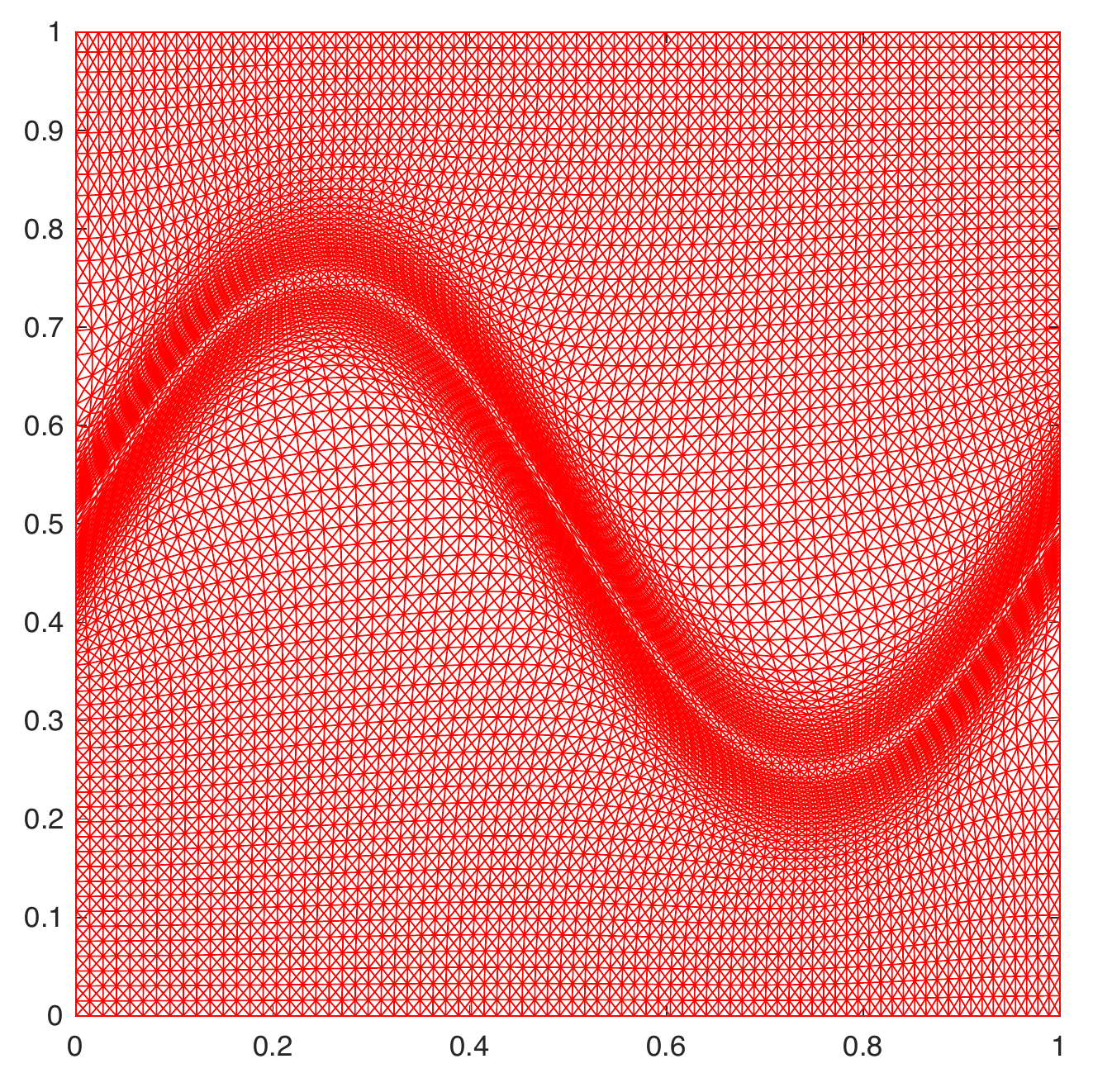}
\end{center}
\centerline{(b) Existing functional}\
\end{minipage}
\hspace{2mm}
\begin{minipage}[t]{2.1in}
\begin{center}
\includegraphics[width=2.1in]{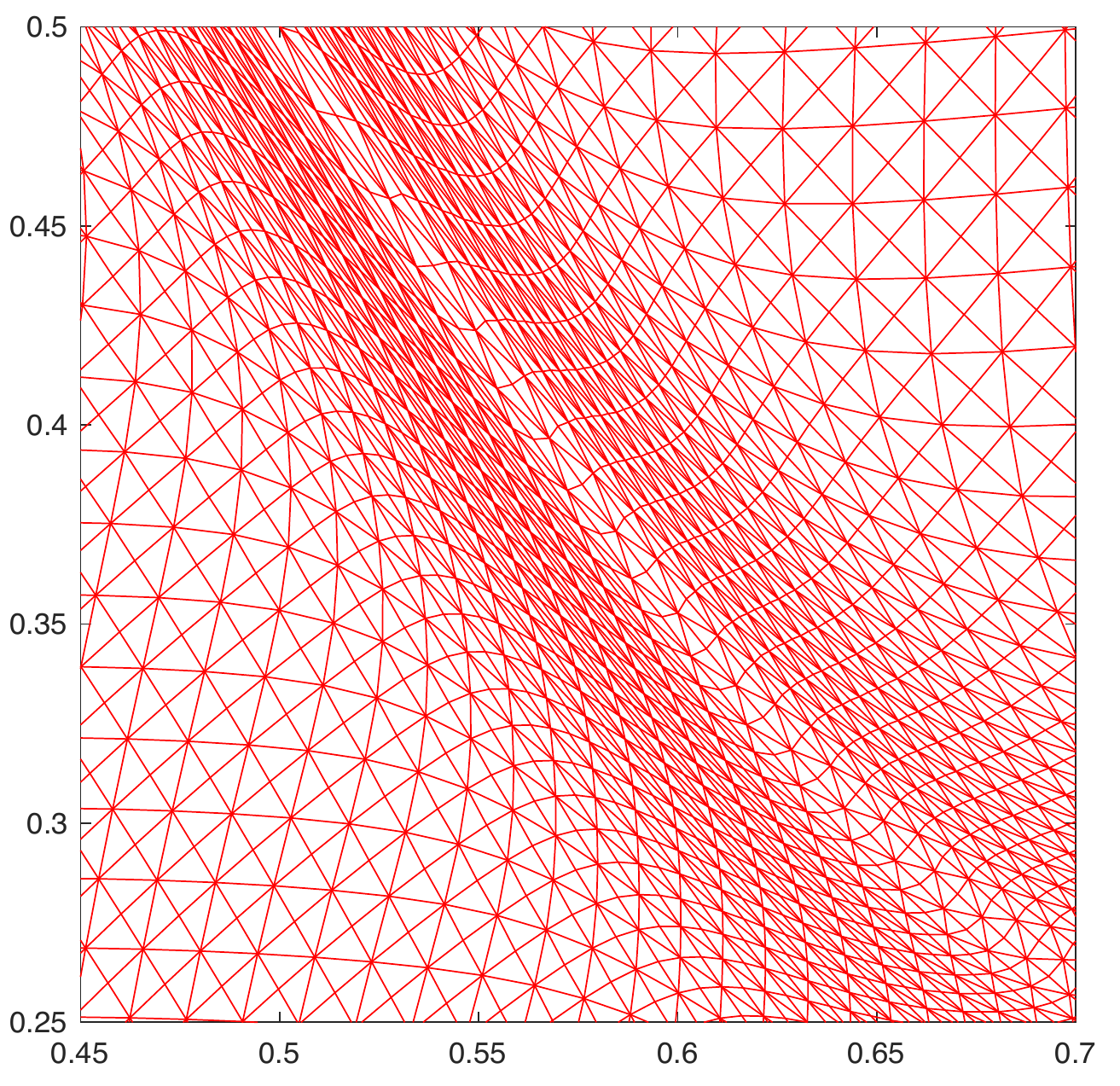}
\end{center}
\end{minipage}
\hspace{2mm}
\begin{minipage}[t]{2.1in}
\begin{center}
\includegraphics[width=2.1in]{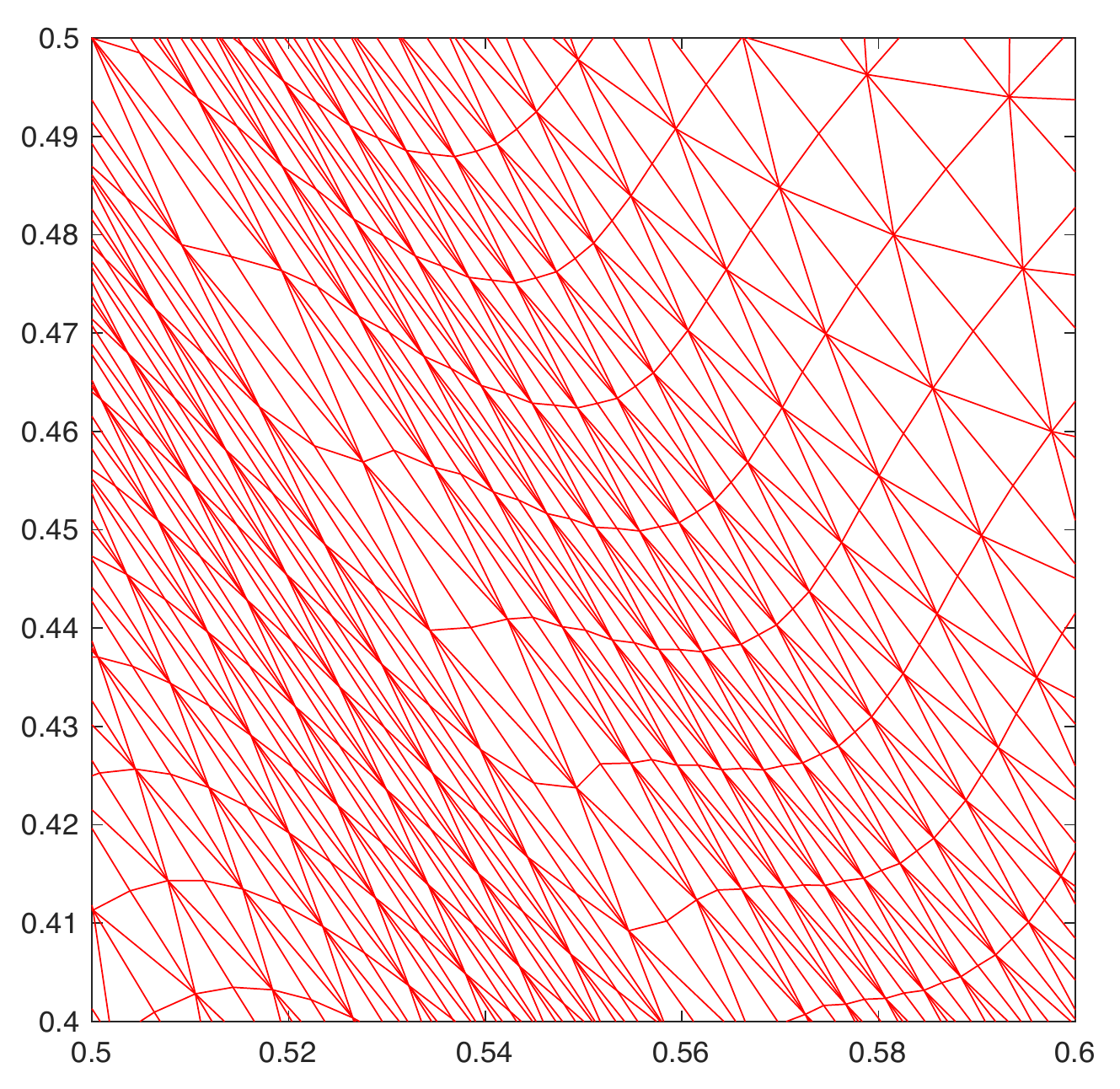}
\end{center}
\end{minipage}
}
\caption{Example~\ref{Exam5.1}. Example meshes (left), close-ups near the inflection point (middle),
and a closer version of the inflection point (right) with $N=25600$.}
\label{fig:exam5.1-mesh}
\end{center}
\end{figure}

\begin{table}[htb]
\caption{Mesh quality measures and the $L^2$ norm of linear interpolation error for Example~\ref{Exam5.1}.}
\begin{center}
\begin{tabular}{| c | c | c | c | c | c |} 
\hline
  Functional & N & $Q_{geo}$ & $Q_{eq}$ & $Q_{ali}$ & error \\
\hline
\multirow{3}{6em}{Existing} & 1600 & 1.684 & 1.065 & 1.041 & 5.563e-3\\ 
& 6400 & 2.000 & 1.071 & 1.042 & 1.219e-3\\
& 25600 & 1.986 & 1.081 & 1.039 & 3.038e-4\\ 
\hline
\multirow{3}{6em} {New} & 1600 &  1.593 & 1.088 & 1.028 & 6.077e-3\\
& 6400 & 1.896 & 1.094 & 1.030 & 1.305e-3 \\ 
& 25600 & 2.019 & 1.091 & 1.030 & 3.138e-4 \\ 
\hline
\end{tabular}
\label{table-5.1}
\end{center}
\end{table}

\begin{figure}[htb]
\begin{center}
\hbox{
\begin{minipage}[t]{2in}
\begin{center}
\begin{tikzpicture}
\node (img) {\includegraphics[width=2.25in]{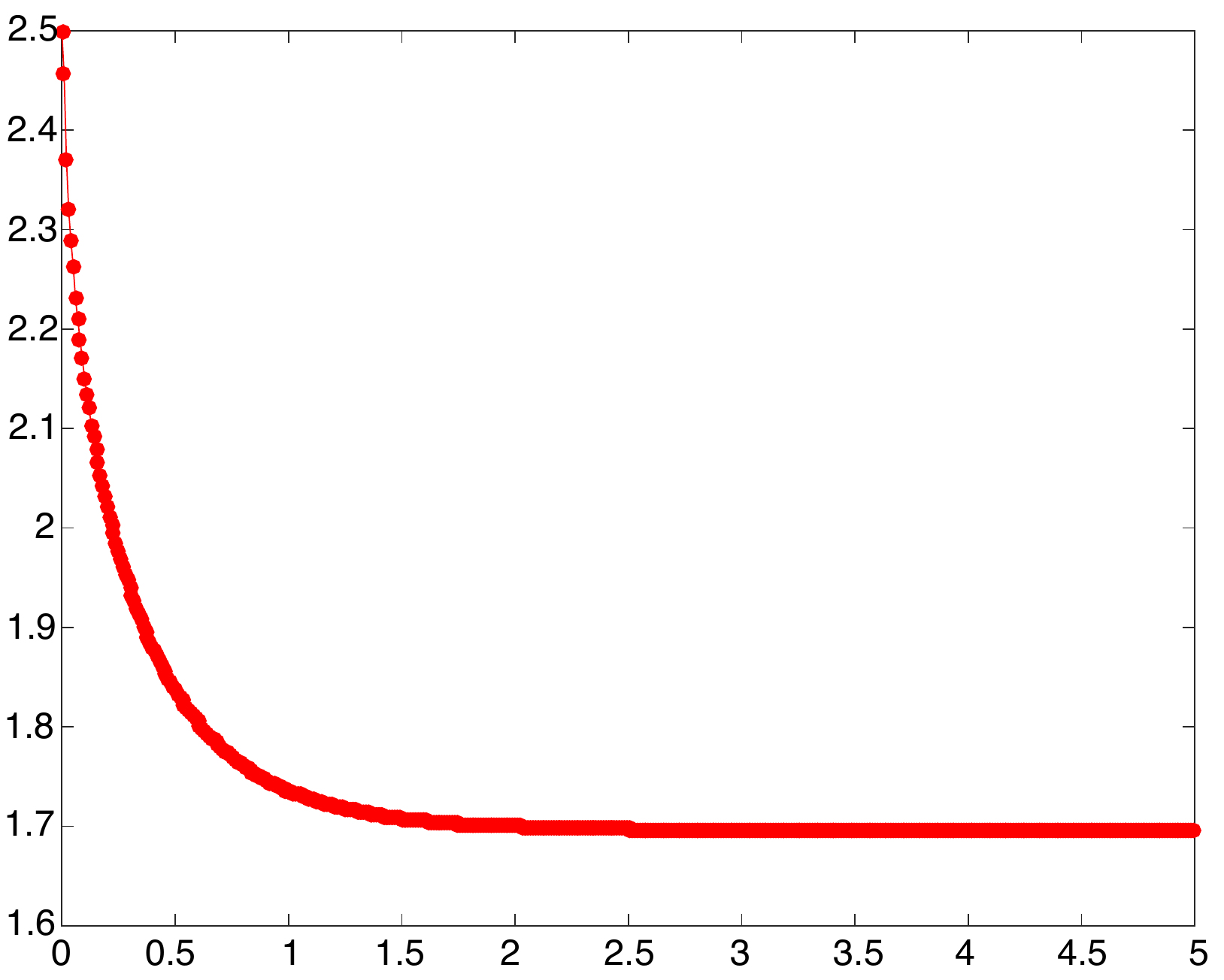}};
\node[below=of img, node distance=0cm, yshift=1cm] {t};
  \node[left=of img, node distance=0cm, rotate=90, anchor=center,yshift=-0.7cm] {$I_h$};
  \end{tikzpicture}
{(a) New functional $I_h$}
\end{center}
\end{minipage}
\hspace{30mm}
\begin{minipage}[t]{2in}
\begin{center}
\begin{tikzpicture}
\node (img) {\includegraphics[width=2.3in]{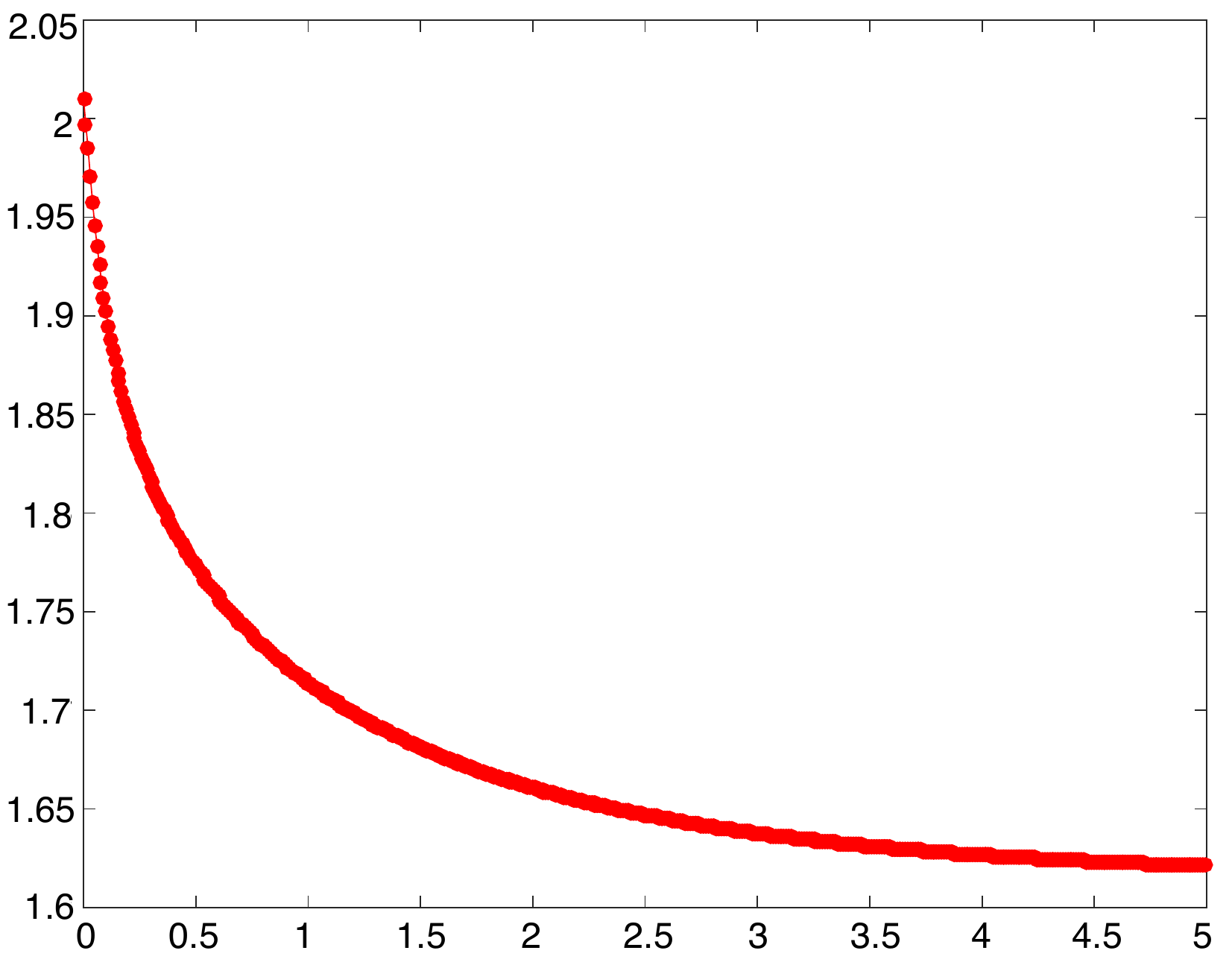}};
\node[below=of img, node distance=0cm, yshift=1cm] {t};
  \node[left=of img, node distance=0cm, rotate=90, anchor=center,yshift=-0.7cm] {$I_h$};
  \end{tikzpicture}
{(b) Existing functional $I_h$}
\end{center}
\end{minipage}
}
\vspace{10mm}
\hbox{
\begin{minipage}[t]{2in}
\begin{center}
\begin{tikzpicture}
\node (img) {\includegraphics[width=2.3in]{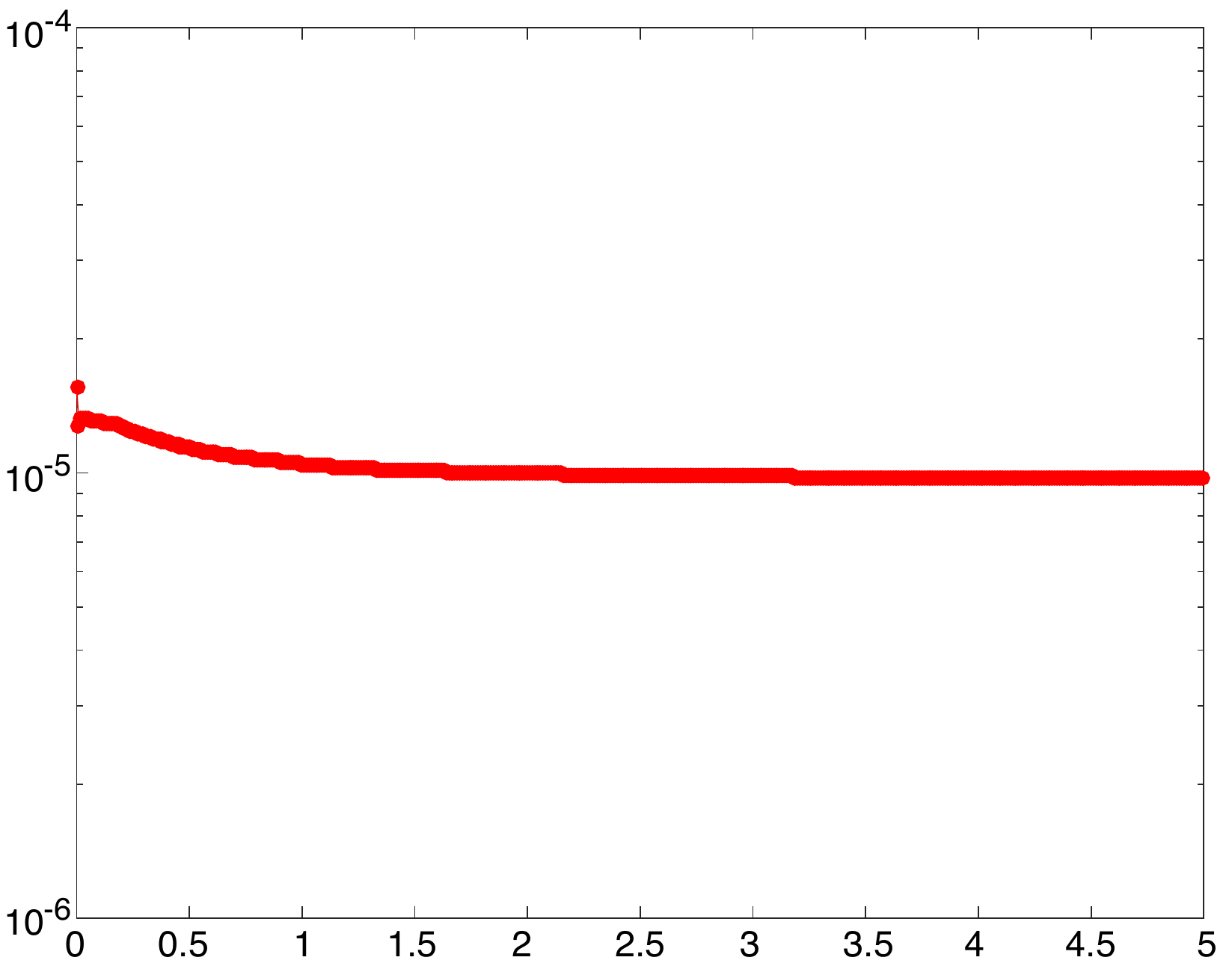}};
\node[below=of img, node distance=0cm, yshift=1cm] {t};
  \node[left=of img, node distance=0cm, rotate=90, anchor=center,yshift=-0.7cm] {$|K|_{\min}$};
  \end{tikzpicture}
{(c) New functional $|K|_{\min}$}
\end{center}
\end{minipage}
\hspace{30mm}
\begin{minipage}[t]{2in}
\begin{center}
\begin{tikzpicture}
\node (img) {\includegraphics[width=2.3in]{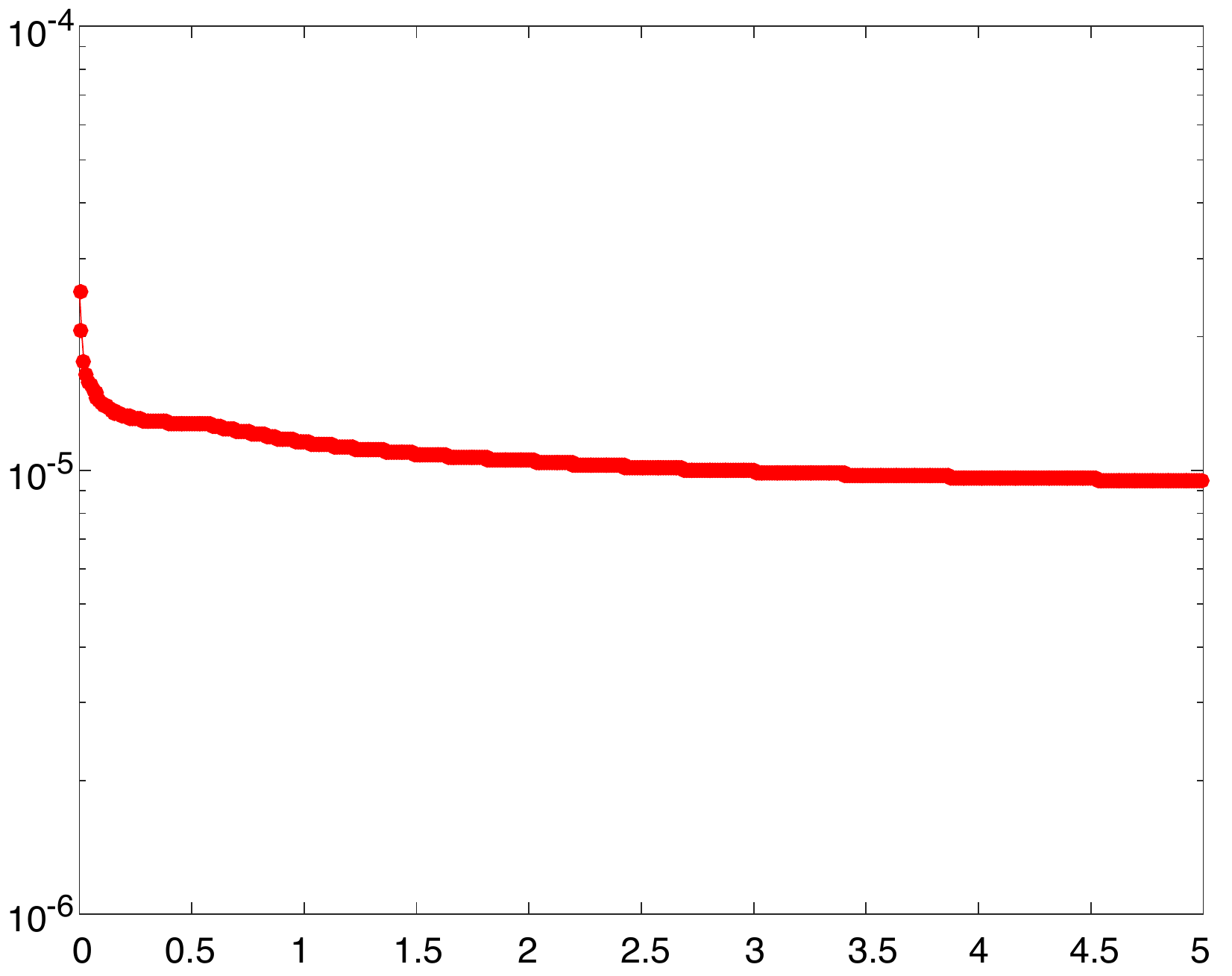}};
\node[below=of img, node distance=0cm, yshift=1cm] {t};
  \node[left=of img, node distance=0cm, rotate=90, anchor=center,yshift=-0.7cm] {$|K|_{\min}$};
  \end{tikzpicture}
{(d) Existing functional $|K|_{\min}$}
\end{center}
\end{minipage}
}
\caption{Example~\ref{Exam5.1}. The energy and minimum element volume are plotted as functions of $t$ with $N=25600$.}
\label{fig:exam5.1-energy}
\end{center}
\end{figure}

While studying Table~\ref{table-5.1}, we can also see that the $Q_{eq}$ quality measure for the new functional
is close to $1$, hence indicating that the mesh associated with the new functional is close to satisfying
the equidistribution condition with respect to $\M$. Therefore, with the alignment and equidistribution
conditions close to being satisfied we can conclude that the mesh is almost uniform under the metric $\M$.
The error value is a good indication that the mesh associated with the new functional is accurate.
In this example, the error associated with the new functional is reasonably low. Moreover, as $N$ increases, the numerical
results show that the error decreases like $\mathcal{O}(N^{-1})$, a second-order convergence rate
in terms of the average element diameter $\bar{h} = 1/\sqrt{N}$. This is consistent with the analysis
of linear interpolation on triangle meshes.

As discussed in Section~\ref{SEC:theory}, theoretically we know that the $I_h$ value is decreasing
and $|K|$ is bonded below.  To see these numerically, we plot $I_h$ and $|K|_{min}$ as functions of $t$
in Fig.~\ref{fig:exam5.1-energy}.  The results are consistent with the theoretical predictions.
Specifically, Fig.~\ref{fig:exam5.1-energy}(a) shows that $I_h$ is decreasing while
Fig.~\ref{fig:exam5.1-energy}(c) suggests that $|K|_{\min}$ is bounded below about $10^{-5}$. It is interesting to observe that Fig.~\ref{fig:exam5.1-energy}(a) shows $I_h$ decreasing faster at the beginning then leveling out more quickly when compared to the existing functional (Fig.~\ref{fig:exam5.1-energy}(b)). This shows that the energy is converging faster for the new functional than for the existing functional.

For comparison purpose, we also show the results obtained with the existing functional in Table~\ref{table-5.1}
and Figs.~\ref{fig:exam5.1-mesh} and \ref{fig:exam5.1-energy}. From these, we see a high correlation.
With respect to the mesh, both are very similar with high concentration near the interface where
the function has large curvature. The quality measures $Q_{geo}$, $Q_{eq}$, and $Q_{ali}$ are very
similar as well.  We further remark that the CPU time for both functionals are almost equivalent, differing at most
by a few seconds. Hence, we can see that the two functionals are very comparable and both seem to work well in
this example. To save space, we do not present here numerical results comparing
(\ref{Huang1}) and (\ref{Huang2}) with other meshing functionals.
The interested reader is referred to \cite{HKR} for additional numerical comparisons.

To show how both functionals perform in a more anisotropic example, we change the constant $30$
in Example \ref{Exam5.1} to $100$ and generate adaptive meshes. In this case, we run to a final time of 0.1.
Fig.~\ref{fig:exam5.1-100mesh} shows the meshes and close-ups. As we can see from studying the meshes, the new functional provides a more adaptive mesh around the region with large curvature. That is, there is a higher concentration of mesh elements that are skew with respect to the Euclidean norm in this region. This is further confirmed by Table~\ref{table-5.1-100} where we see $Q_{geo}\approx1.894$ for the new functional and $Q_{geo}\approx1.279$ for the existing functional ($N=25600$).  It is also observed from $Q_{eq}$ and $Q_{ali}$ in Table~$\ref{table-5.1-100}$ that the mesh associated with the new functional is slightly more uniform with respect to the metric tensor $\mathbb{M}$.
Moreover, the interpolation error for the new functional is about half that of the existing functional for $N=25600$. 
Overall, both functionals handle this more anisotropic example well and comparably.

\begin{figure}[htb]
\begin{center}
\hbox{
\begin{minipage}[t]{1.05in}
\mbox{ }
\end{minipage}
\begin{minipage}[t]{2.1in}
\begin{center}
\includegraphics[width=2.1in]{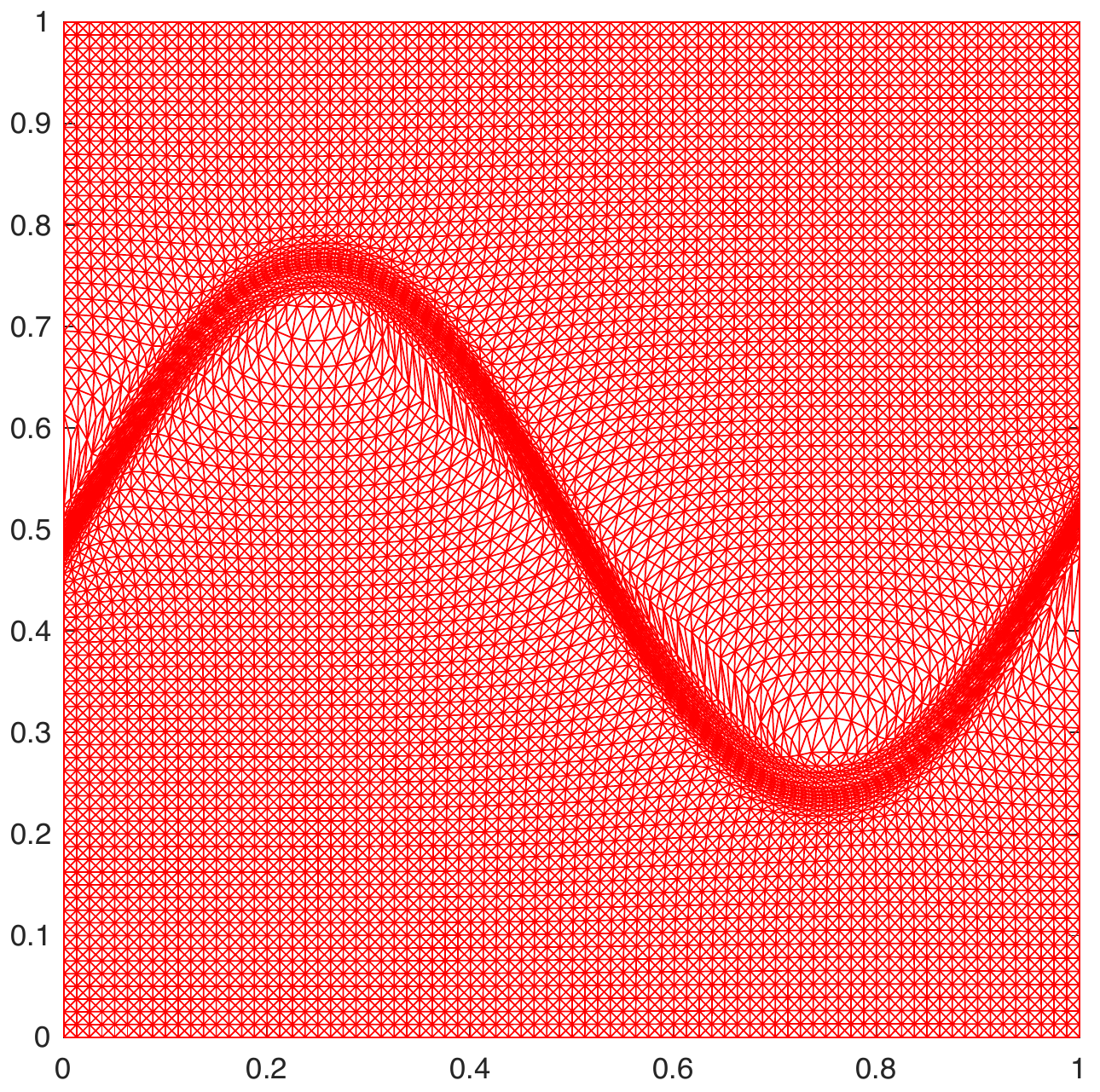}
\end{center}
\centerline{(a) New functional}\
\end{minipage}
\hspace{2mm}
\begin{minipage}[t]{2.1in}
\begin{center}
\includegraphics[width=2.1in]{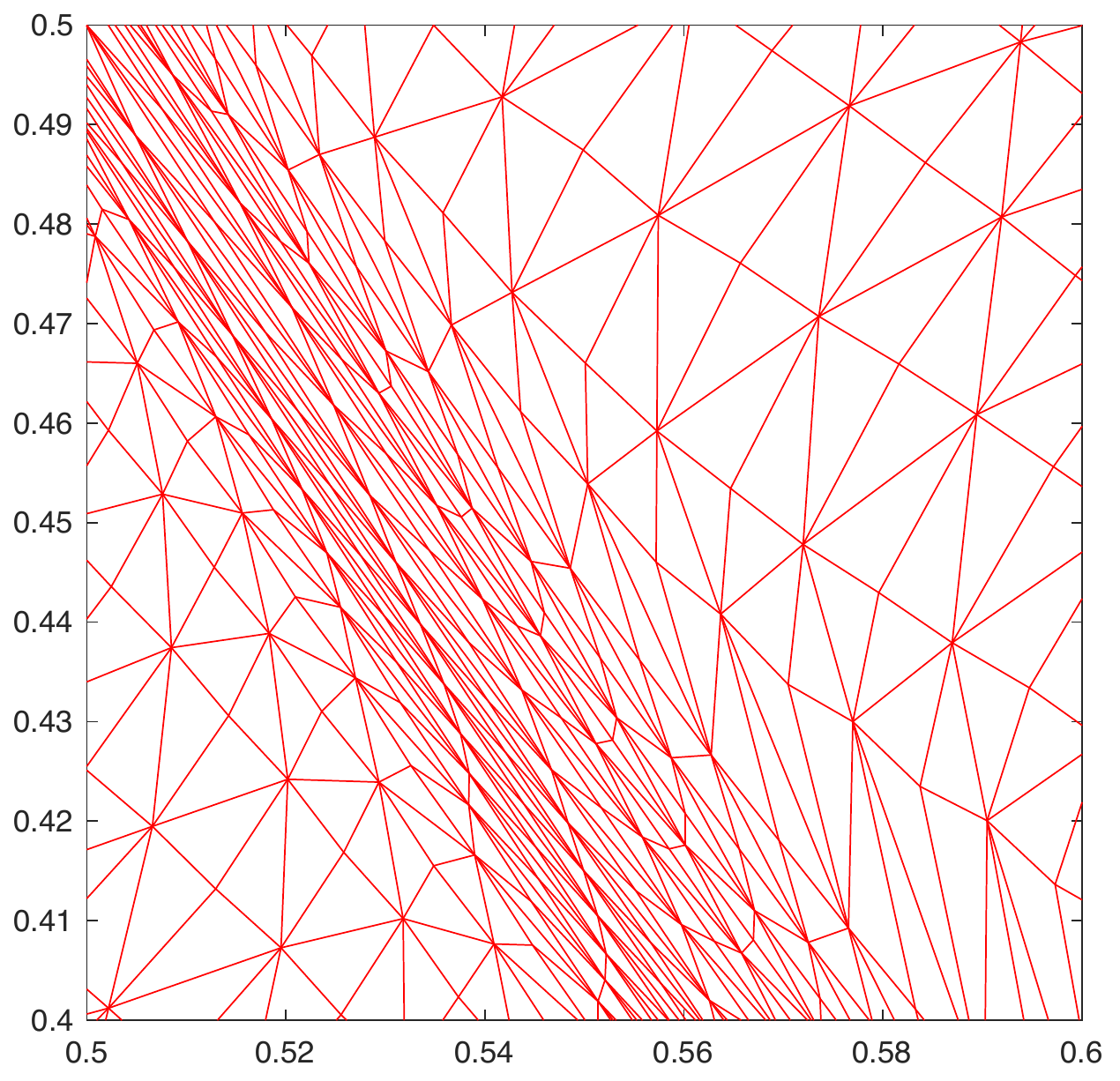}
\end{center}
\end{minipage}
}
\hbox{
\begin{minipage}[t]{1.05in}
\mbox{ }
\end{minipage}
\begin{minipage}[t]{2.1in}
\begin{center}
\includegraphics[width=2.1in]{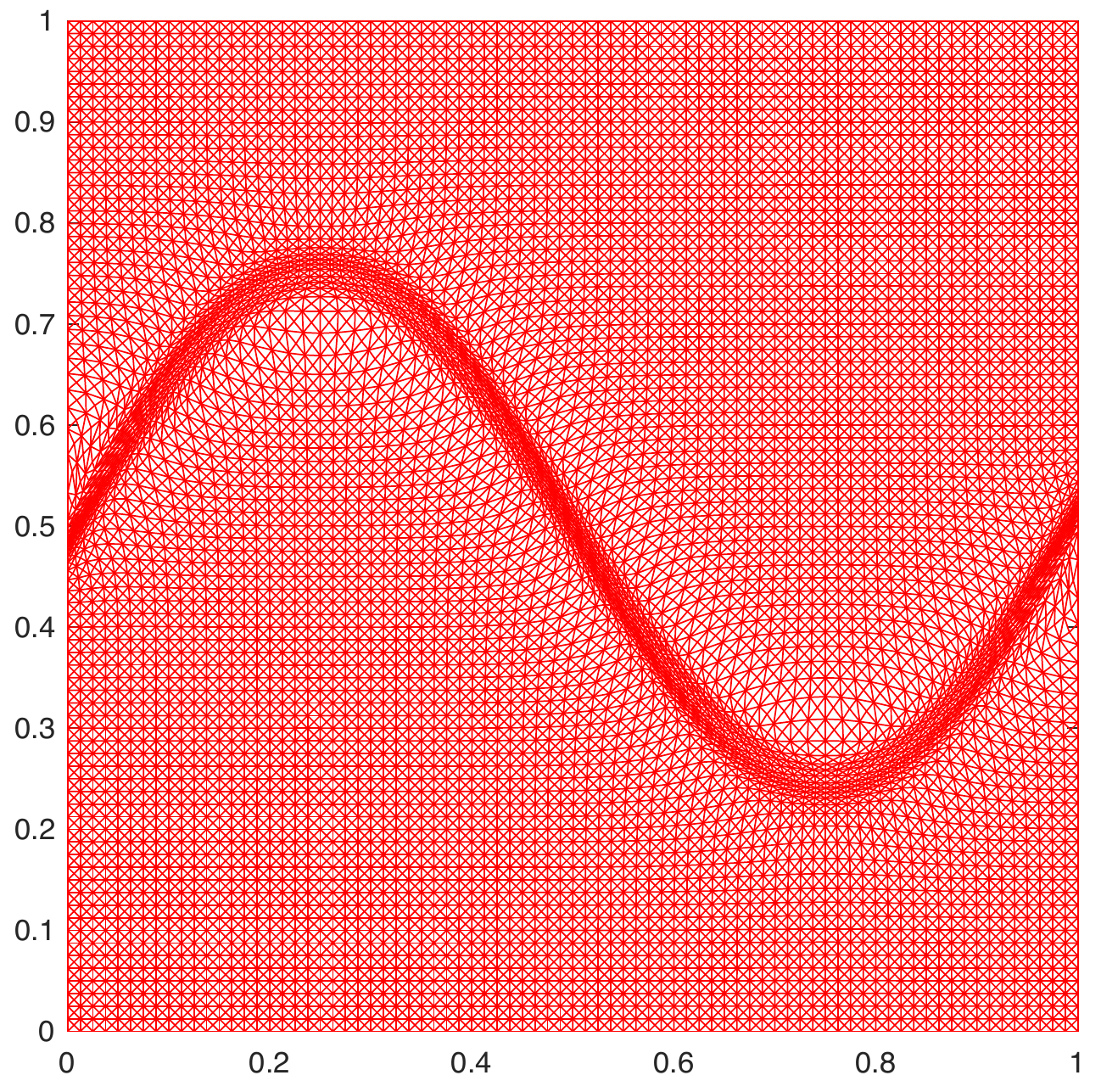}
\end{center}
\centerline{(b) Existing functional}\
\end{minipage}
\hspace{2mm}
\begin{minipage}[t]{2.1in}
\begin{center}
\includegraphics[width=2.1in]{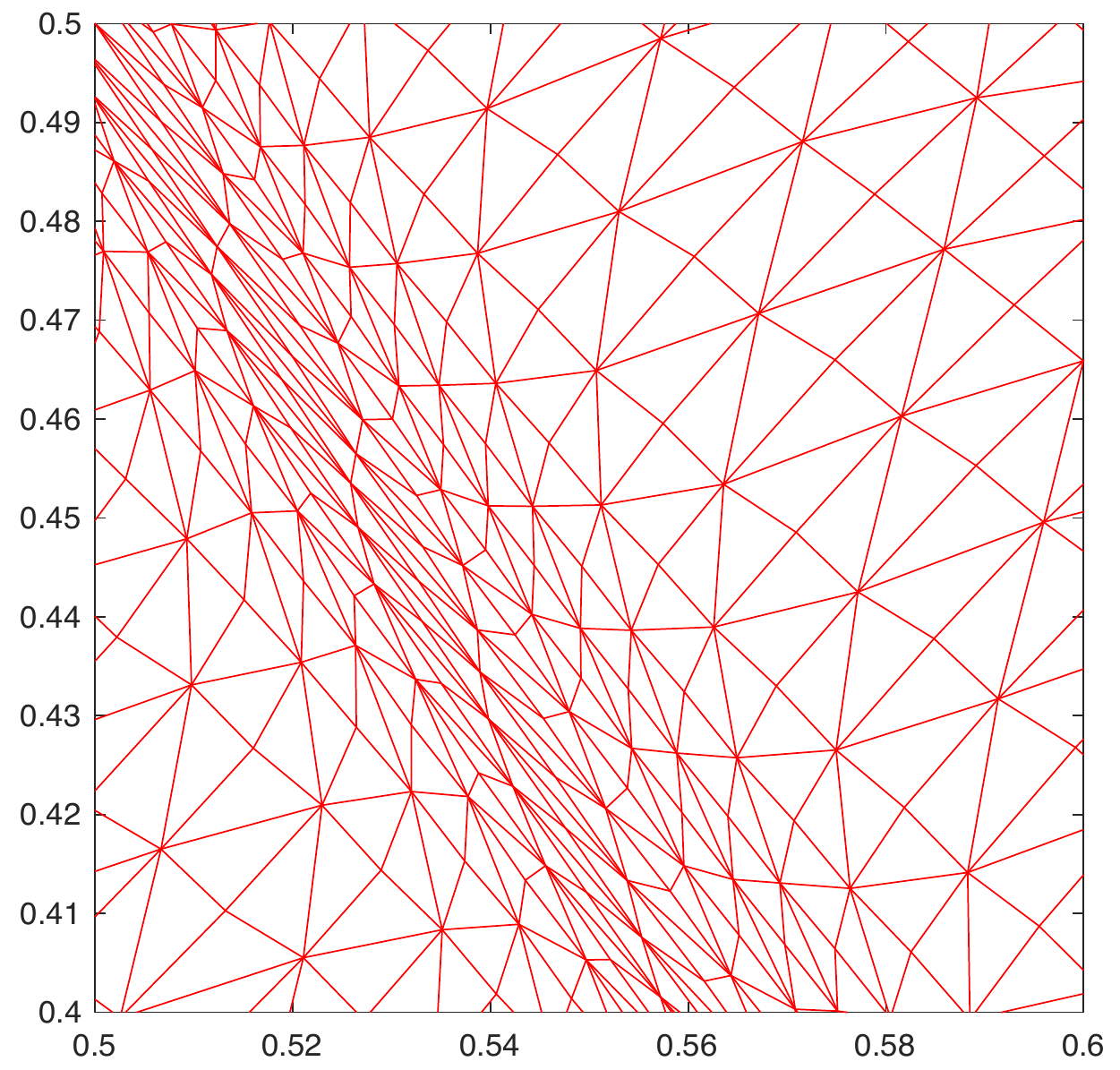}
\end{center}
\end{minipage}
}
\caption{Example~\ref{Exam5.1} with more anisotropic features. Adaptive meshes (left) and close-ups near
the inflection point (right) with $N=25600$.}
\label{fig:exam5.1-100mesh}
\end{center}
\end{figure}

\begin{table}[htb]
\caption{Mesh quality measures and the $L^2$ norm of linear interpolation error for Example~\ref{Exam5.1} with more anisotropic features.}
\begin{center}
\begin{tabular}{| c | c | c | c | c | c |} 
\hline
  Functional & N & $Q_{geo}$ & $Q_{eq}$ & $Q_{ali}$ & error \\
\hline
\multirow{3}{6em}{Existing} & 1600 & 1.626 & 1.155 & 1.034 & 1.807e-2\\ 
& 6400 & 1.548 & 1.312 & 1.058 & 3.942e-3\\
& 25600 & 1.279 & 1.553 & 1.107 & 2.462e-3\\ 
\hline
\multirow{3}{6em} {New} & 1600 &  2.059 & 1.148 & 1.031 & 1.232e-2\\
& 6400 &  2.203 & 1.249 & 1.028 & 2.616e-3 \\ 
& 25600 &  1.894 & 1.436 & 1.067 & 1.261e-3 \\ 
\hline
\end{tabular}
\label{table-5.1-100}
\end{center}
\end{table}

\end{exam}


\begin{exam}
\label{Exam5.2}
In this example, we generate adaptive meshes for a five sphere figure modeled by
\begin{align*}
   u(x,y) = &\tanh\left(30\left(X^2+Y^2-\frac{1}{8}\right)\right)
              + \tanh\left(30\left((X-0.5)^2+(Y-0.5)^2-\frac{1}{8}\right)\right)\\
     & + \tanh\left(30\left((X-0.5)^2+(Y+0.5)^2-\frac{1}{8}\right)\right)\\
     &+ \tanh\left(30\left((X+0.5)^2+(Y-0.5)^2-\frac{1}{8}\right)\right)\\
     & +\tanh\left(30\left((X+0.5)^2+(Y+0.5)^2-\frac{1}{8}\right)\right) ,
\end{align*}
where $X = -2 + 4x$ and $Y = -2 + 4y$. We integrate the MMPDE (\ref{MMPDEx}) to $t = 0.5$.

Fig.~\ref{fig:exam5.2-mesh} shows the meshes and close-ups of both functionals for this example.
Studying the figure we see that the new functional provides a mesh with accurate shape and size
adaptation. This can be further confirmed by the quality measures and the linear interpolation error
given in Table~\ref{table-5.2}. One may notice that the mesh has smaller values of $Q_{geo}$
and thus is less skew than those in the previous example. This may be due to the fact that the function
in this example is more isotropic than that in the previous example. Moreover, the linear interpolation error
behaves like $\mathcal{O}(N^{-1})$, showing a second-order convergence rate.

\begin{figure}[htb]
\begin{center}
\hbox{
\begin{minipage}[t]{2.1in}
\begin{center}
\includegraphics[width=2.1in]{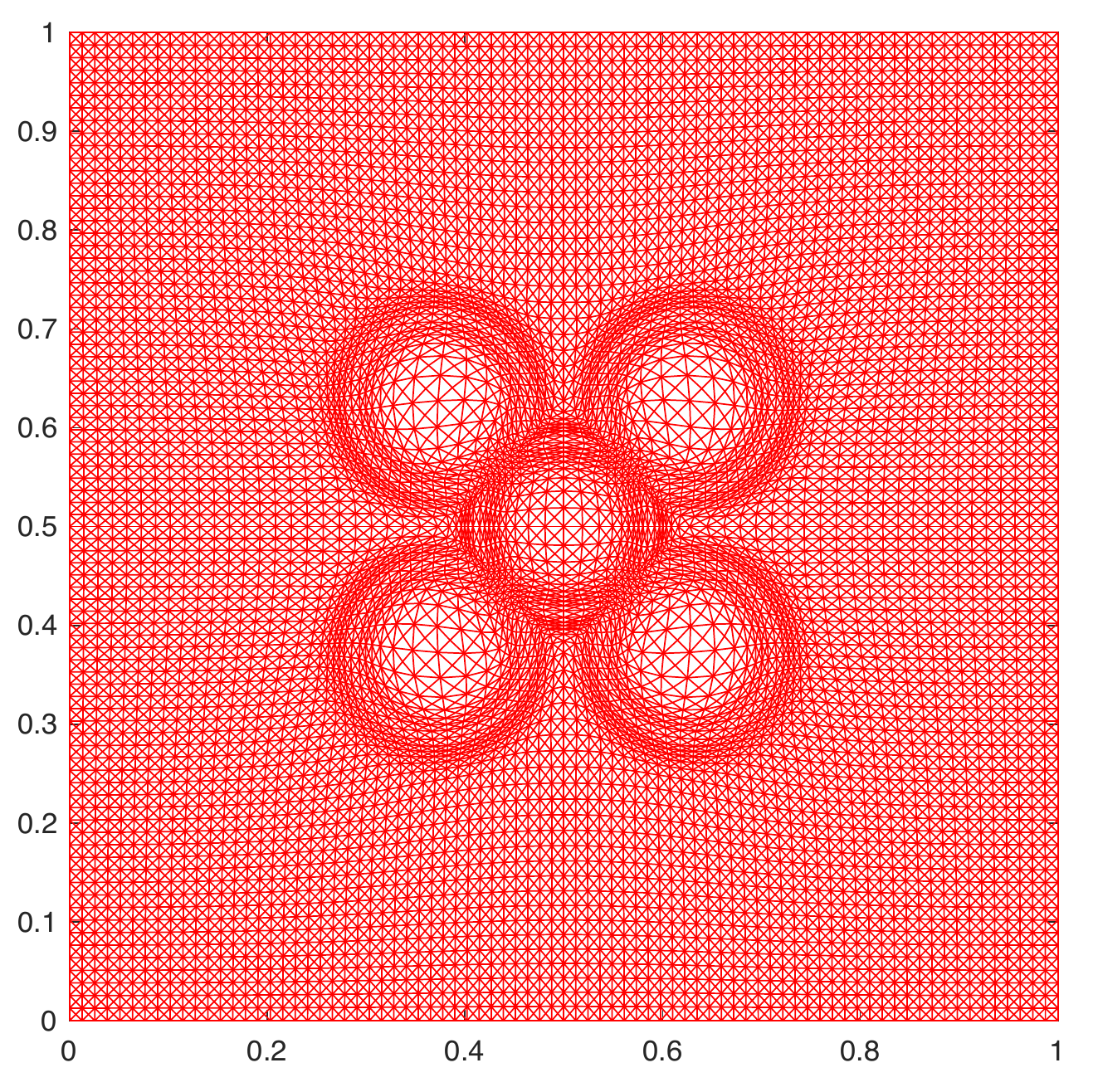}
\end{center}
\centerline{(a) New functional}
\end{minipage}
\hspace{2mm}
\begin{minipage}[t]{2.1in}
\begin{center}
\includegraphics[width=2.1in]{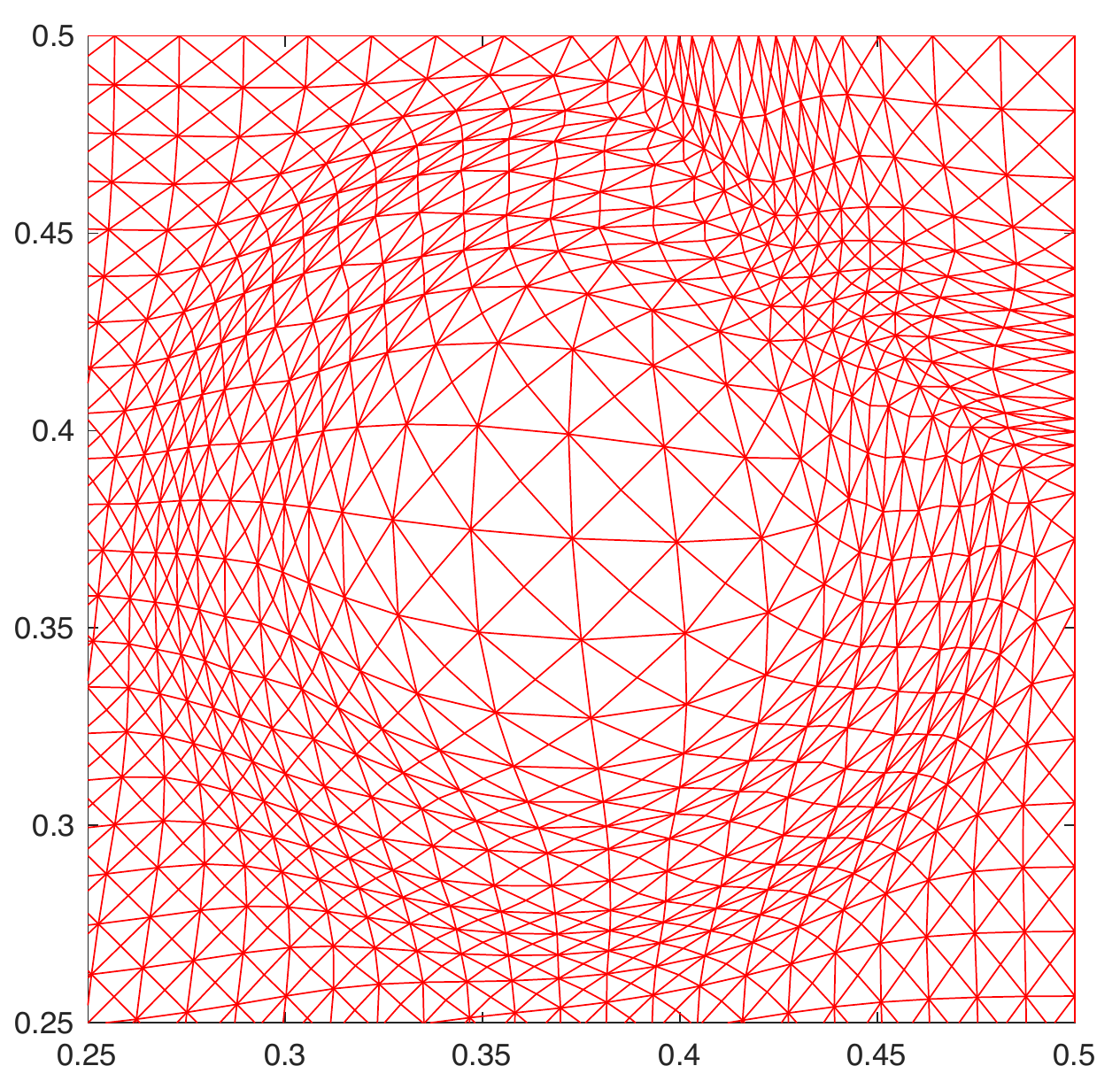}
\end{center}
\end{minipage}
\hspace{2mm}
\begin{minipage}[t]{2.1in}
\begin{center}
\includegraphics[width=2.1in]{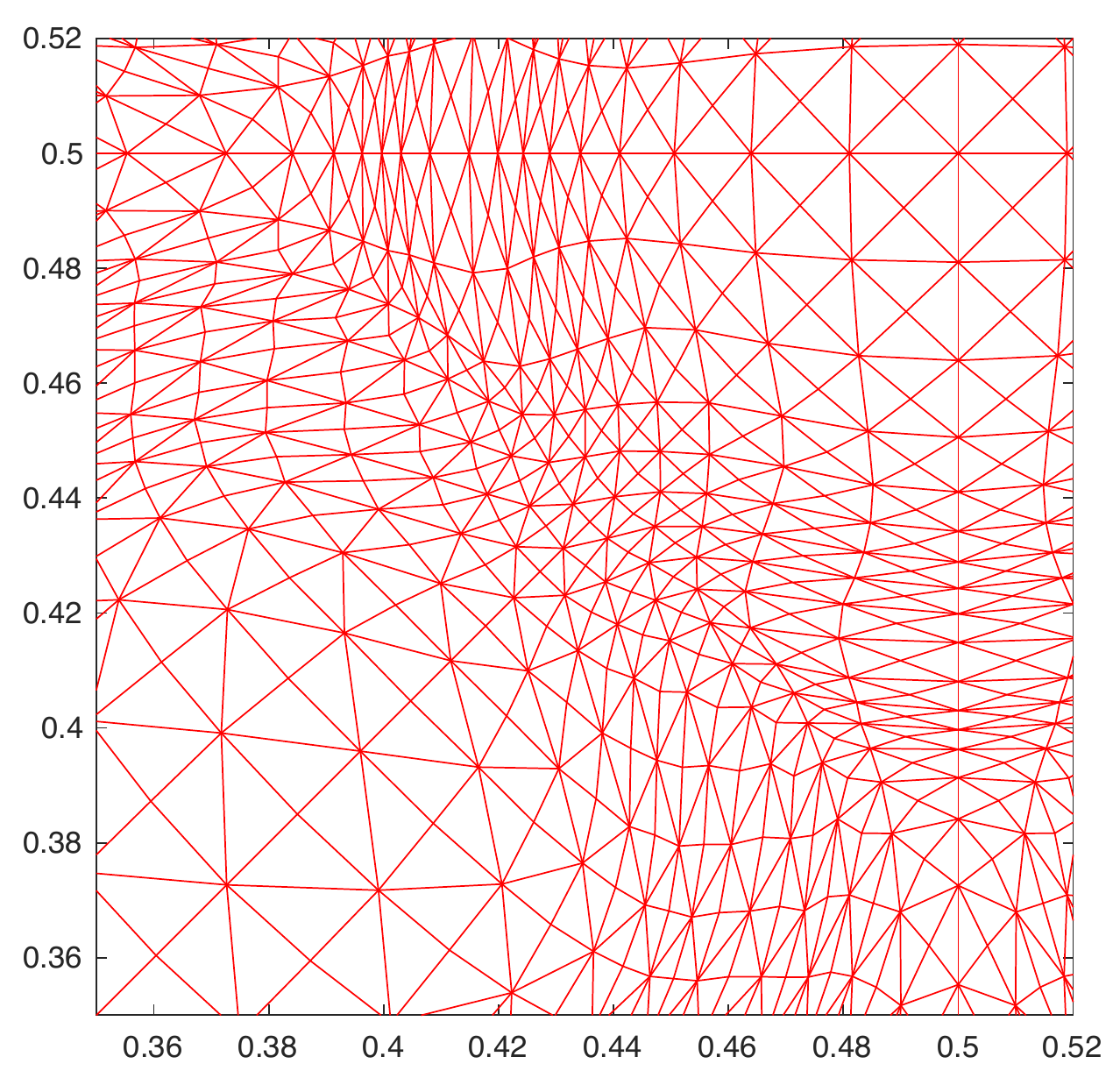}
\end{center}
\end{minipage}
}
\centering
\hbox{
\begin{minipage}[t]{2.1in}
\begin{center}
\includegraphics[width=2.1in]{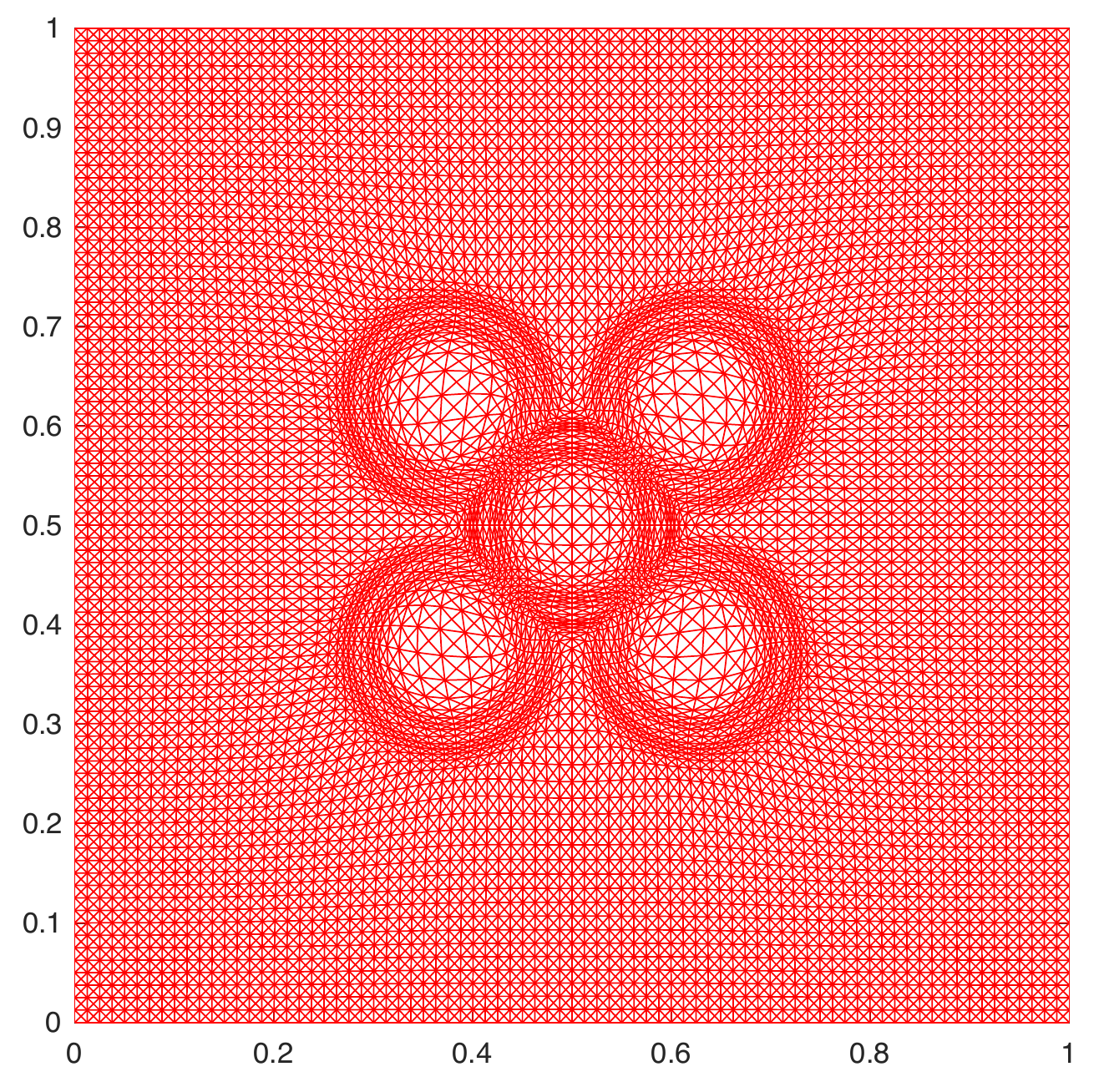}
\end{center}
\centerline{(a) Existing functional}
\end{minipage}
\hspace{2mm}
\begin{minipage}[t]{2.1in}
\begin{center}
\includegraphics[width=2.1in]{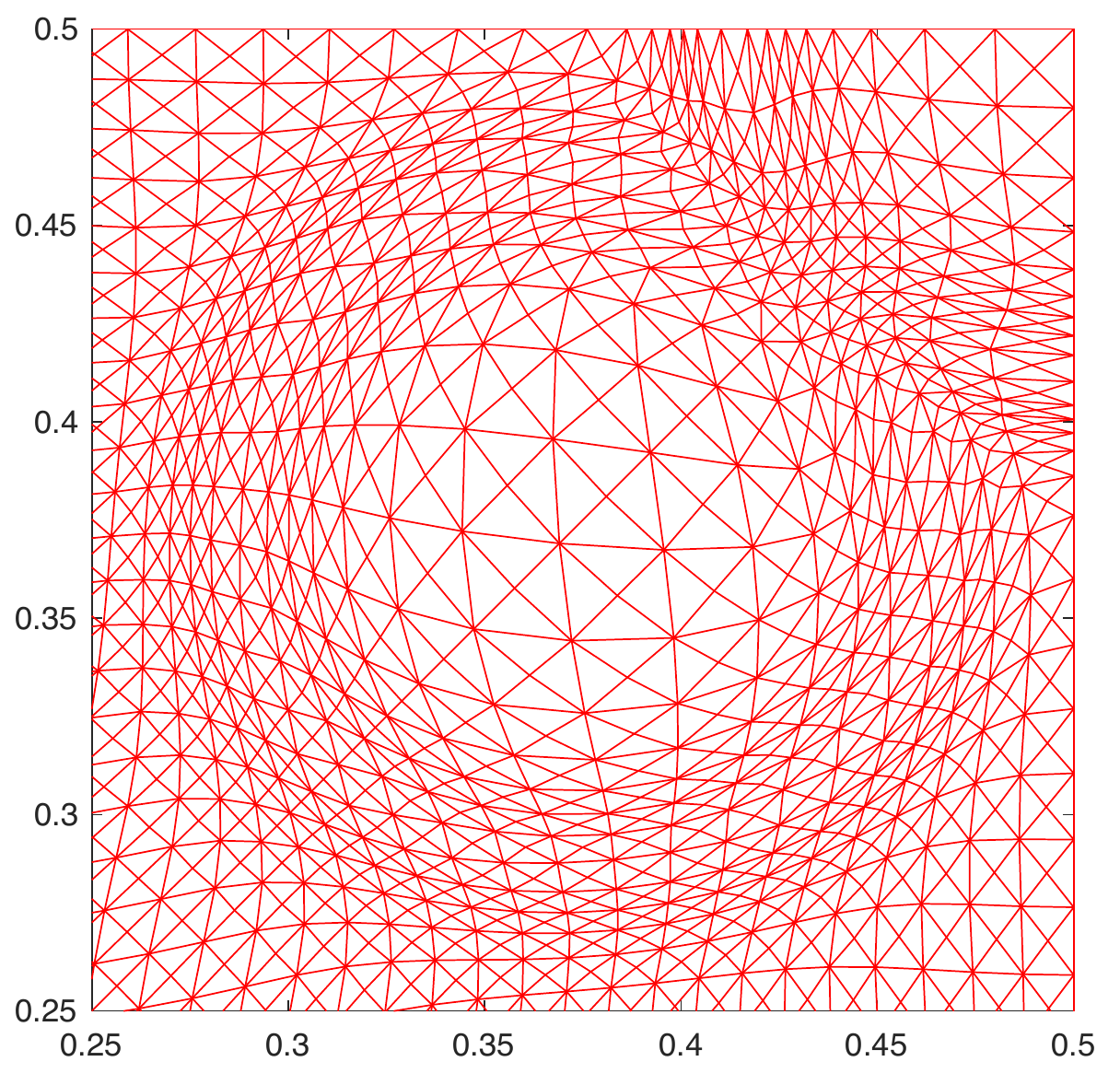}
\end{center}
\end{minipage}
\hspace{2mm}
\begin{minipage}[t]{2.1in}
\begin{center}
\includegraphics[width=2.1in]{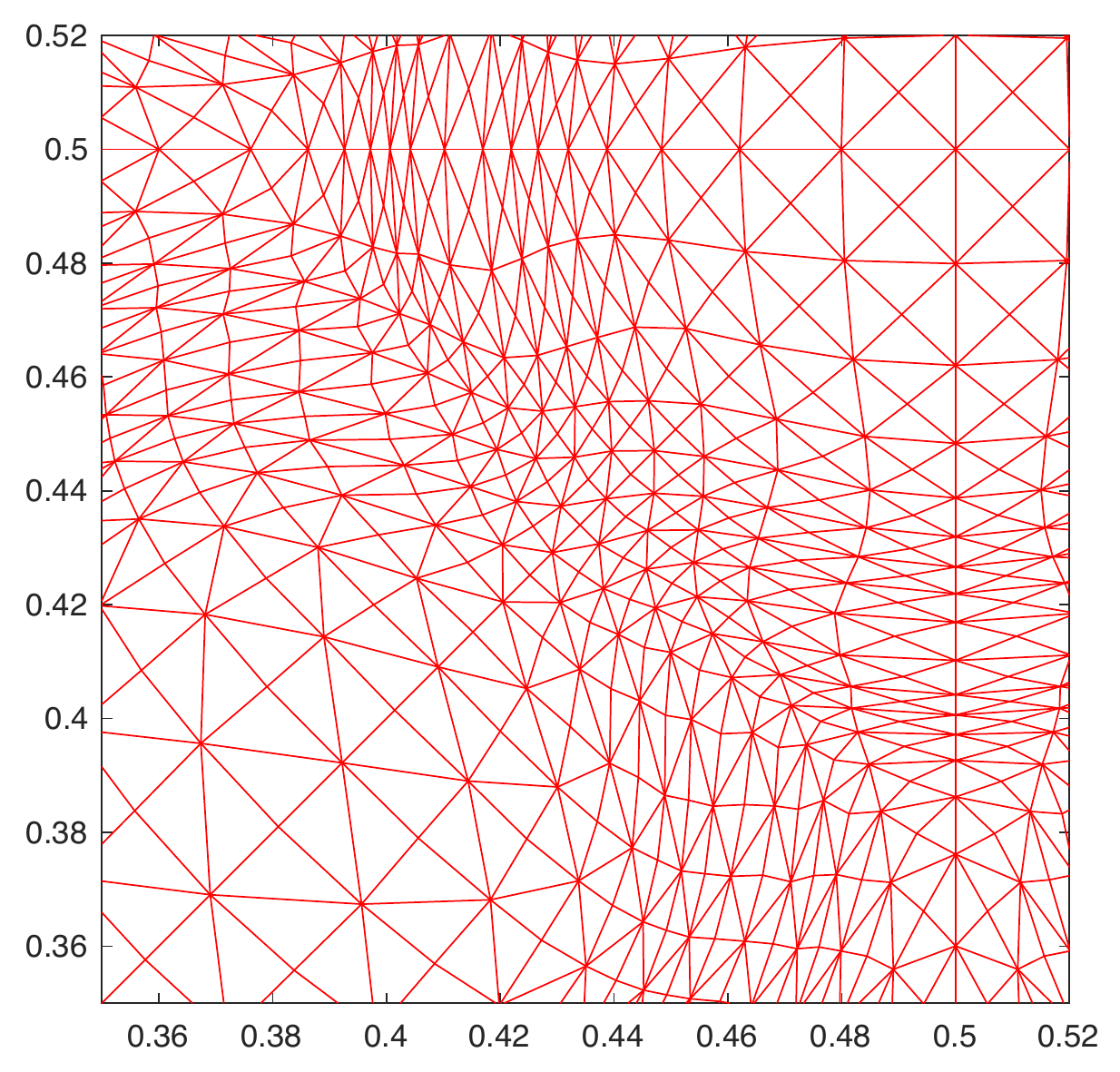}
\end{center}
\end{minipage}
}\
\caption{Example~\ref{Exam5.2}. Example meshes (left), close-ups near the circle meeting the boundary layer (middle),
and a closer version of the circle meeting the boundary layer (right) with $N=25600$.}
\label{fig:exam5.2-mesh}
\end{center}
\end{figure}

Fig.~\ref{fig:exam5.2-energy} shows the energy and minimum volume of the elements as functions of time.
One can see that $I_h$ is decreasing and converging faster for the new functional than for the existing functional, and that $|K|_{\min}$ is bounded by about $10^{-5}$.
Moreover, the results and performance of the new functional are similar to those with the existing functional.
\end{exam}

\begin{table}[htb]
\caption{Mesh quality measures and the $L^2$ norm of linear interpolation error for Example~\ref{Exam5.2}.}
 \begin{center}
\begin{tabular}{| c | c | c | c | c | c |} 
\hline
  Functional & N & $Q_{geo}$ & $Q_{eq}$ & $Q_{ali}$ & error \\
\hline
\multirow{3}{6em}{Existing} & 1600 & 1.051 & 1.134 & 1.056 & 6.954e-2\\ 
& 6400 & 1.094 & 1.231 & 1.057 & 1.326e-2\\
& 25600 & 1.122 & 1.342 & 1.040 & 3.068e-3 \\ 
\hline
\multirow{3}{6em} {New} & 1600 & 1.031 & 1.188 & 1.026 & 6.946e-2 \\
& 6400 & 1.076 & 1.300 & 1.030 & 1.794e-2  \\
& 25600 & 1.137 & 1.370  & 1.030 & 3.310e-3 \\
\hline
\end{tabular}
\label{table-5.2}
\end{center}
\end{table}

\begin{figure}[htb]
\begin{center}
\hbox{
\begin{minipage}[t]{2in}
\begin{center}
\begin{tikzpicture}
\node (img) {\includegraphics[width=2.25in]{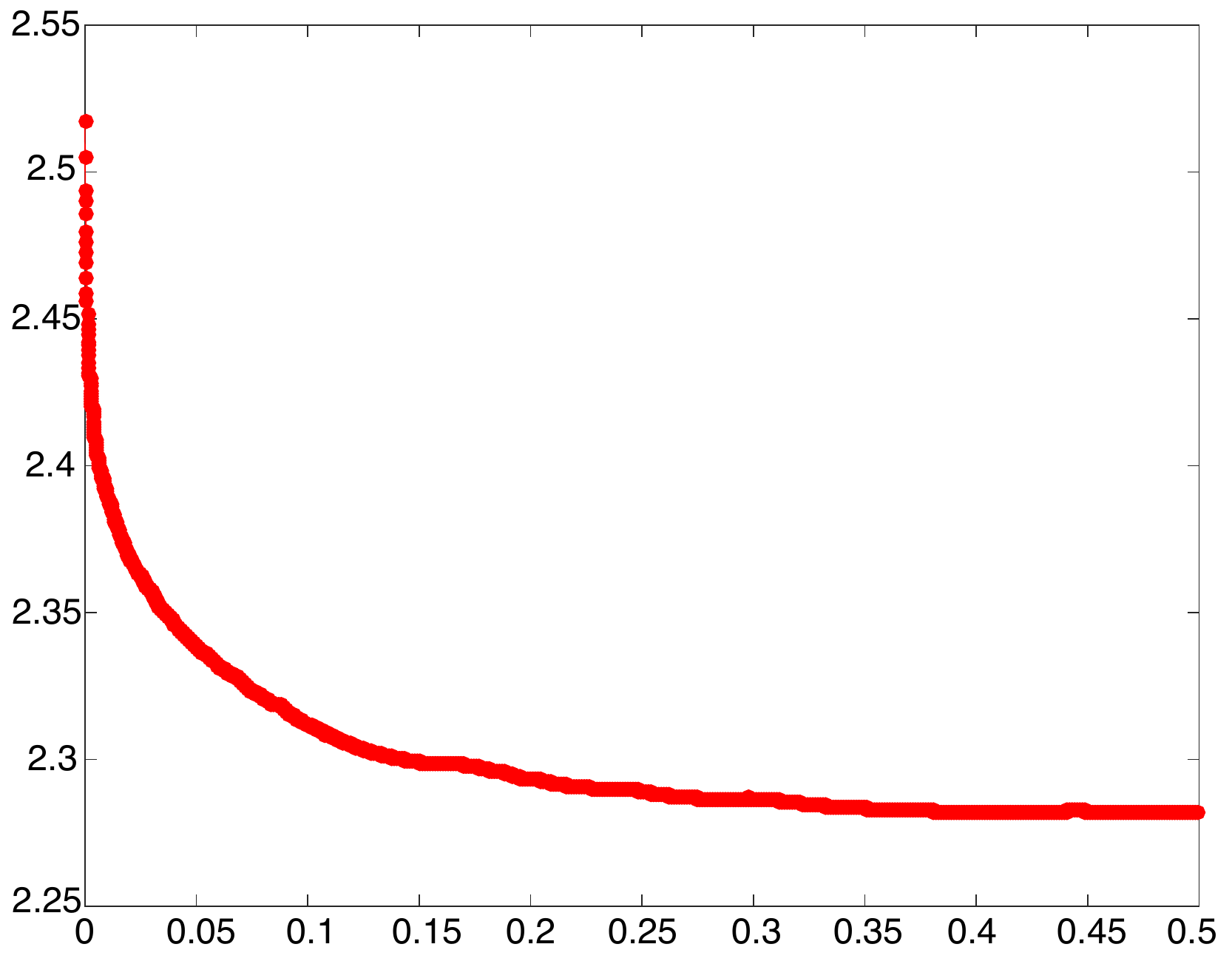}};
\node[below=of img, node distance=0cm, yshift=1cm] {t};
  \node[left=of img, node distance=0cm, rotate=90, anchor=center,yshift=-0.7cm] {$I_h$};
  \end{tikzpicture}
{(a) New functional $I_h$}
\end{center}
\end{minipage}
\hspace{30mm}
\begin{minipage}[t]{2in}
\begin{center}
\begin{tikzpicture}
\node (img) {\includegraphics[width=2.25in]{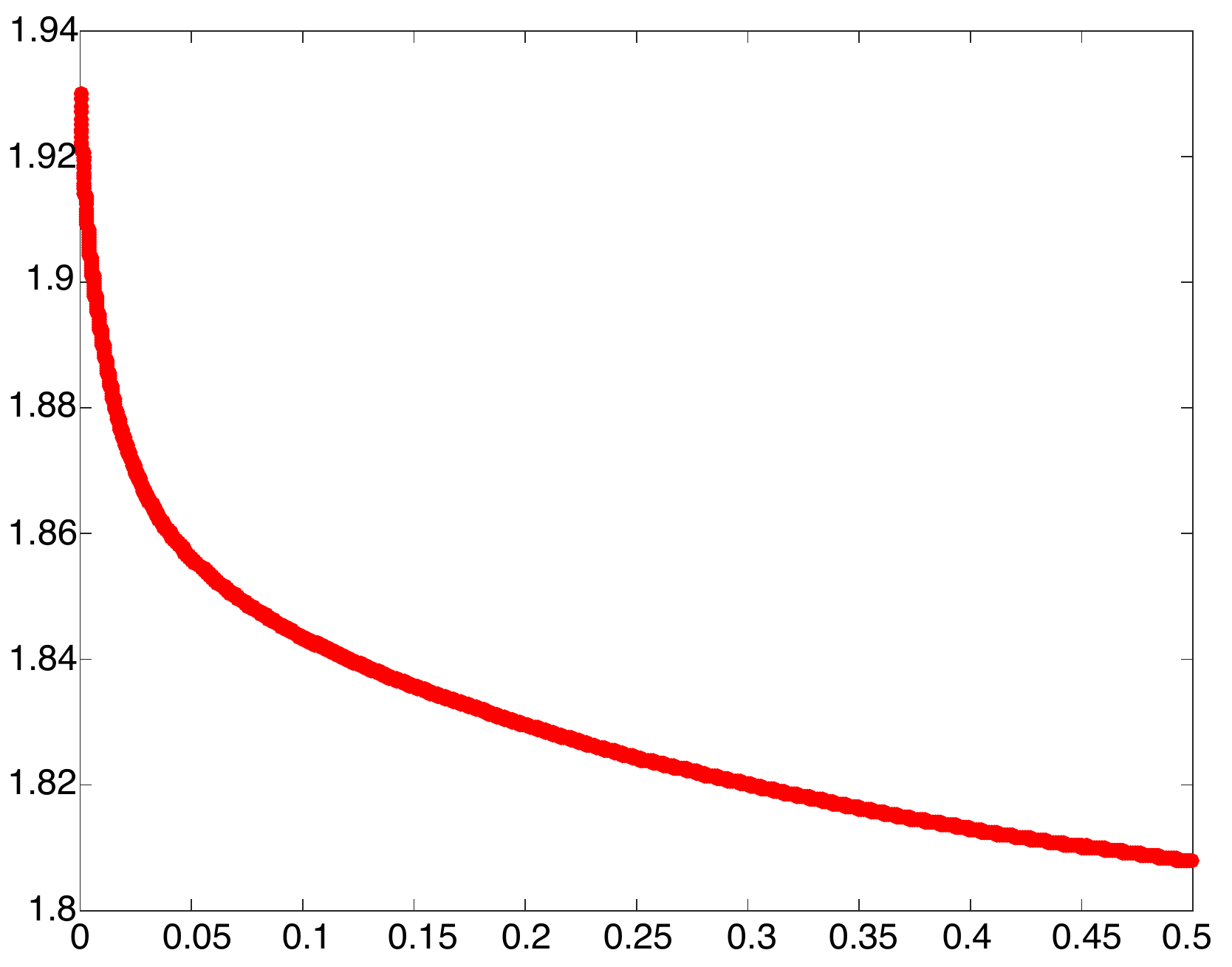}};
\node[below=of img, node distance=0cm, yshift=1cm] {t};
  \node[left=of img, node distance=0cm, rotate=90, anchor=center,yshift=-0.7cm] {$I_h$};
  \end{tikzpicture}
{(b) Existing functional $I_h$}
\end{center}
\end{minipage}
}
\vspace{10mm}
\hbox{
\begin{minipage}[t]{2in}
\begin{center}
\begin{tikzpicture}
\node (img) {\includegraphics[width=2.3in]{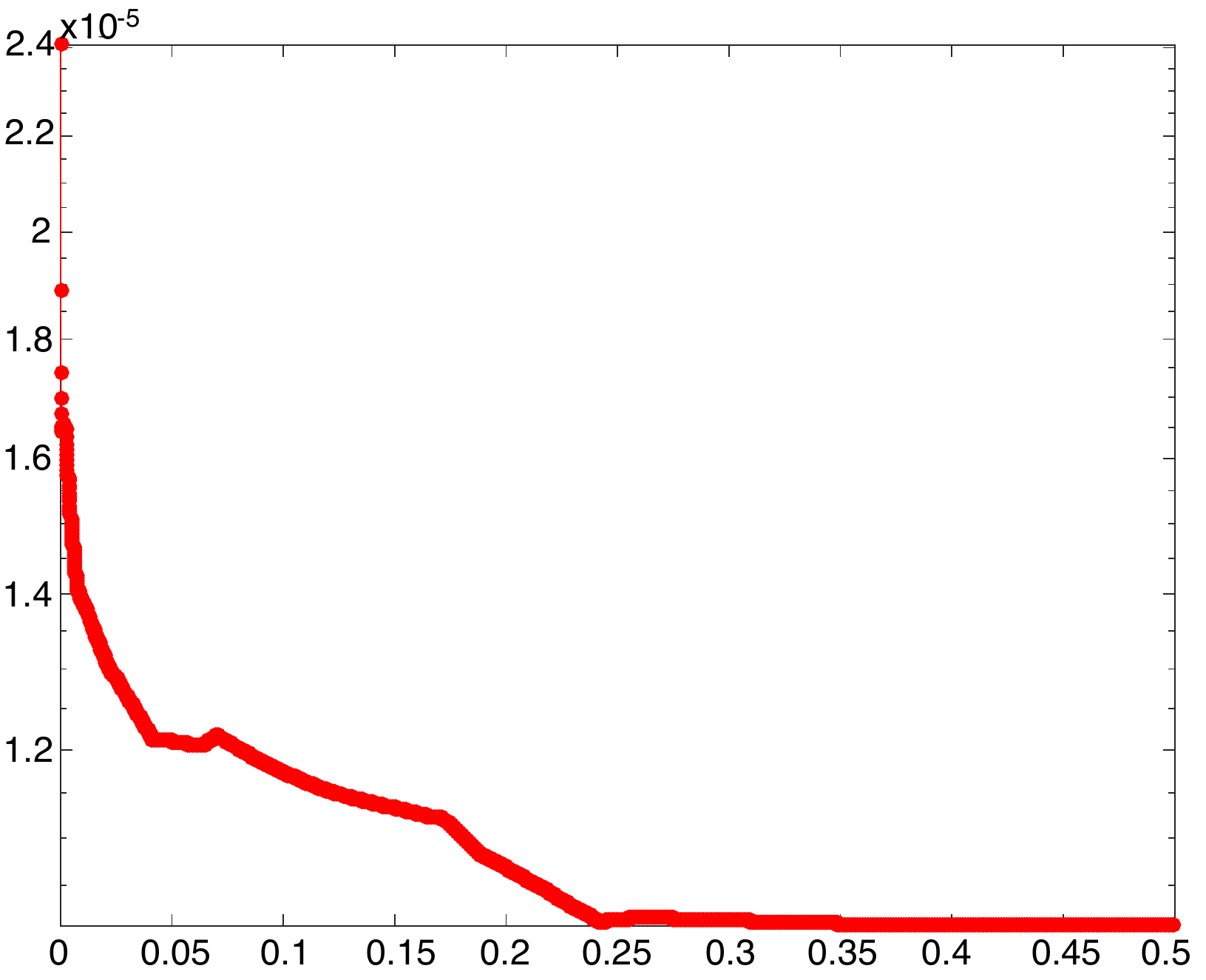}};
\node[below=of img, node distance=0cm, yshift=1cm] {t};
  \node[left=of img, node distance=0cm, rotate=90, anchor=center,yshift=-0.7cm] {$|K|_{\min}$};
  \end{tikzpicture}
{(c) New functional $|K|_{\min}$}
\end{center}
\end{minipage}
\hspace{30mm}
\begin{minipage}[t]{2in}
\begin{center}
\begin{tikzpicture}
\node (img) {\includegraphics[width=2.3in]{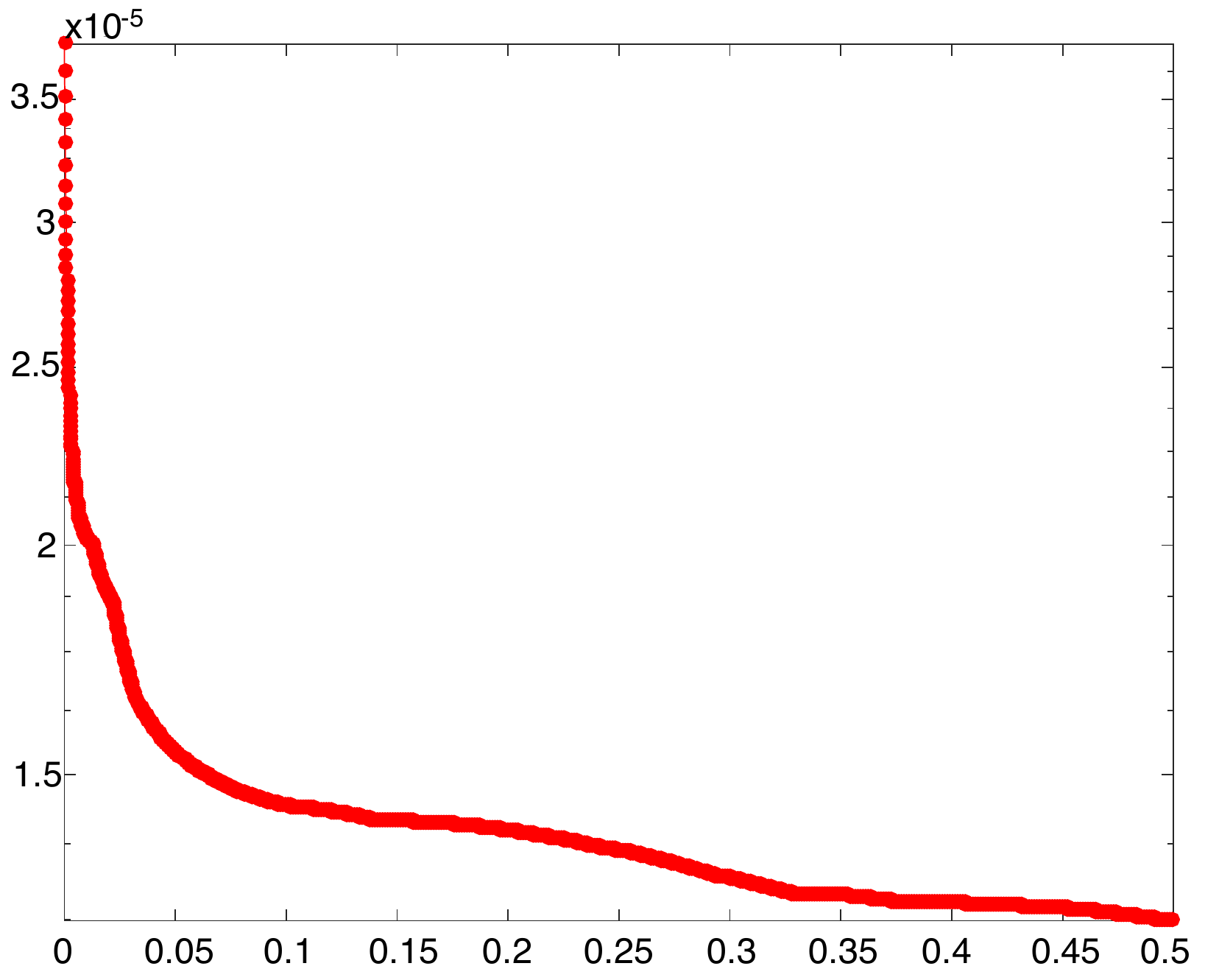}};
\node[below=of img, node distance=0cm, yshift=1cm] {t};
  \node[left=of img, node distance=0cm, rotate=90, anchor=center,yshift=-0.7cm] {$|K|_{\min}$};
  \end{tikzpicture}
{(d) Existing functional $|K|_{\min}$}
\end{center}
\end{minipage}
}
\caption{Example~\ref{Exam5.2}. The energy and minimum element volume are plotted as functions of $t$ with $N=25600$.}
\label{fig:exam5.2-energy}
\end{center}
\end{figure}


\begin{exam}
\label{Exam5.3}
In the final example, we solve the initial-boundary value problem of a special case of Burgers' equation
\[
u_t =10^{-3} \Delta u - u u_x - u u_y, \quad \text{ in } \Omega = (-1,1) \times(-1,1)
\]
subject to a homogeneous boundary condition and the initial condition
\[
u(x,y,0) = e^{-36.8414 (x^2+y^2)}, \quad \text{ in } \Omega .
\]
The partial differential equation is discretized in space using linear finite elements and in time using
the fifth-order Radau IIA method \cite{HW96}. It is solved with the mesh equation in an alternating manner \cite{HR}.
For the following results, we start at $t=0.25$ and run to a final time of $t=1.25$.  

The meshes and close-ups for this example are given in Fig.~\ref{fig:burgers}. Studying the figure we see that the new functional mesh is much more adaptive when compared to the existing functional mesh.  The mesh associated with the new functional provides good shape and size adaptation. As seen in the close-ups, the concentration of mesh elements in the region with large curvature is high which, as we have seen in Examples \ref{Exam5.1} and \ref{Exam5.2}, is consistent with the Hessian based metric tensor. Moreover, the elements for the new functional are much more skew (with respect to the Euclidean metric) in the regions with larger curvature which is confirmed in Table~\ref{table-burgers}
with $Q_{geo}\approx 17.01$. 
  
\begin{figure}[htb]
\begin{center}
\hbox{
\begin{minipage}[t]{2.1in}
\begin{center}
\includegraphics[width=2.1in]{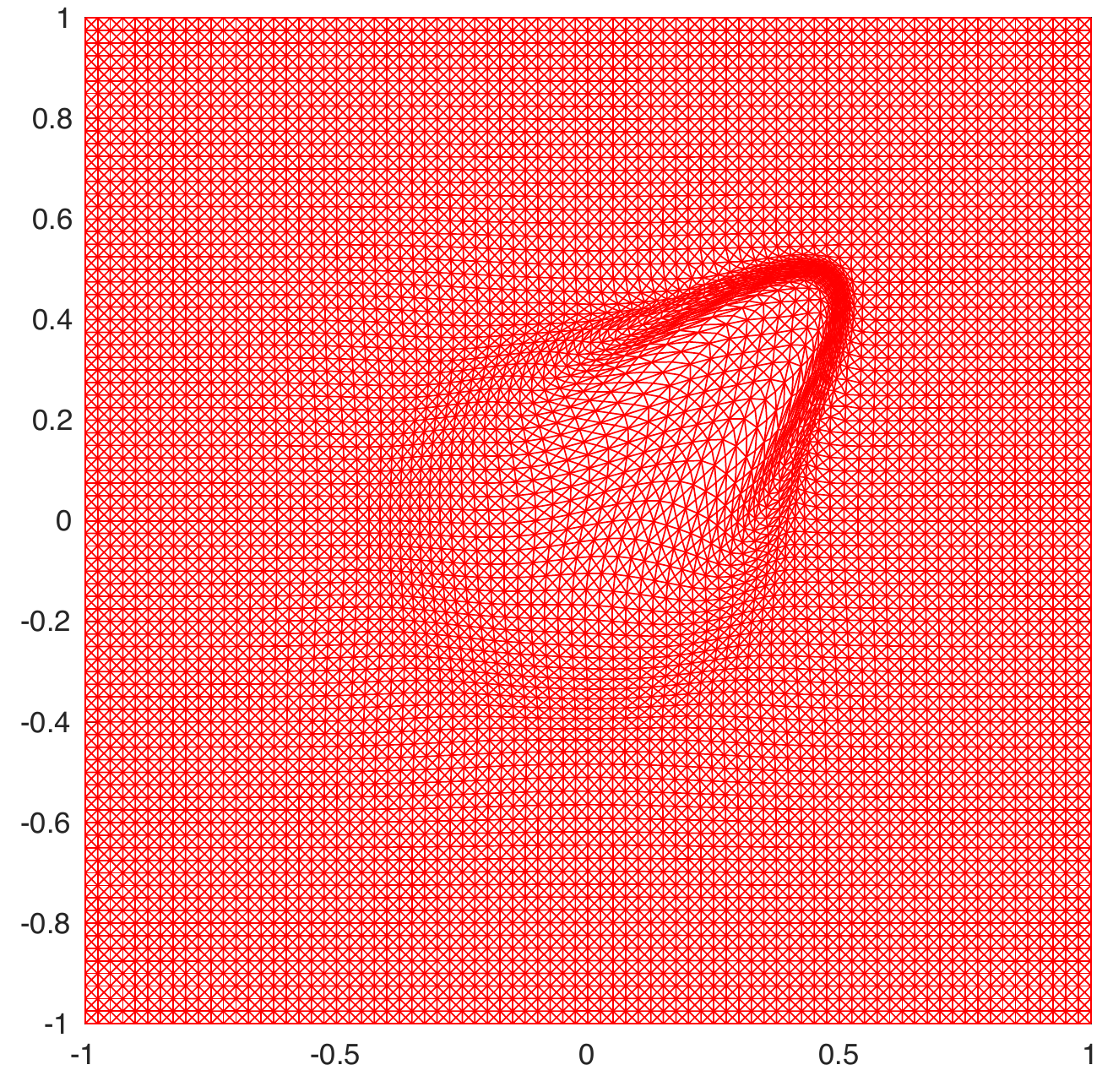}
\end{center}
\centerline{(a) New functional}\
\end{minipage}
\hspace{2mm}
\begin{minipage}[t]{2.1in}
\begin{center}
\includegraphics[width=2.1in]{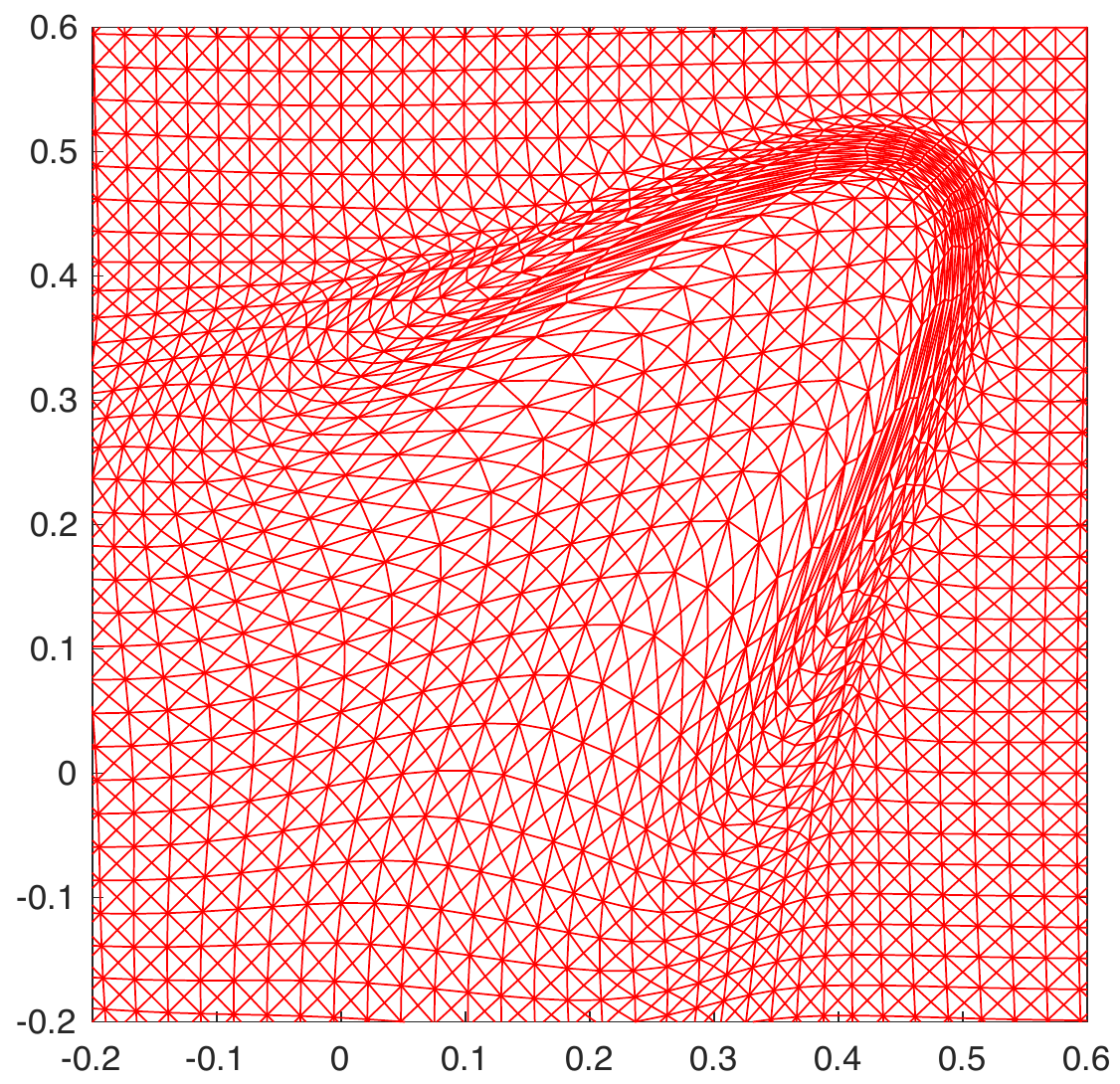}
\end{center}
\end{minipage}
\hspace{2mm}
\begin{minipage}[t]{2.1in}
\begin{center}
\includegraphics[width=2.1in]{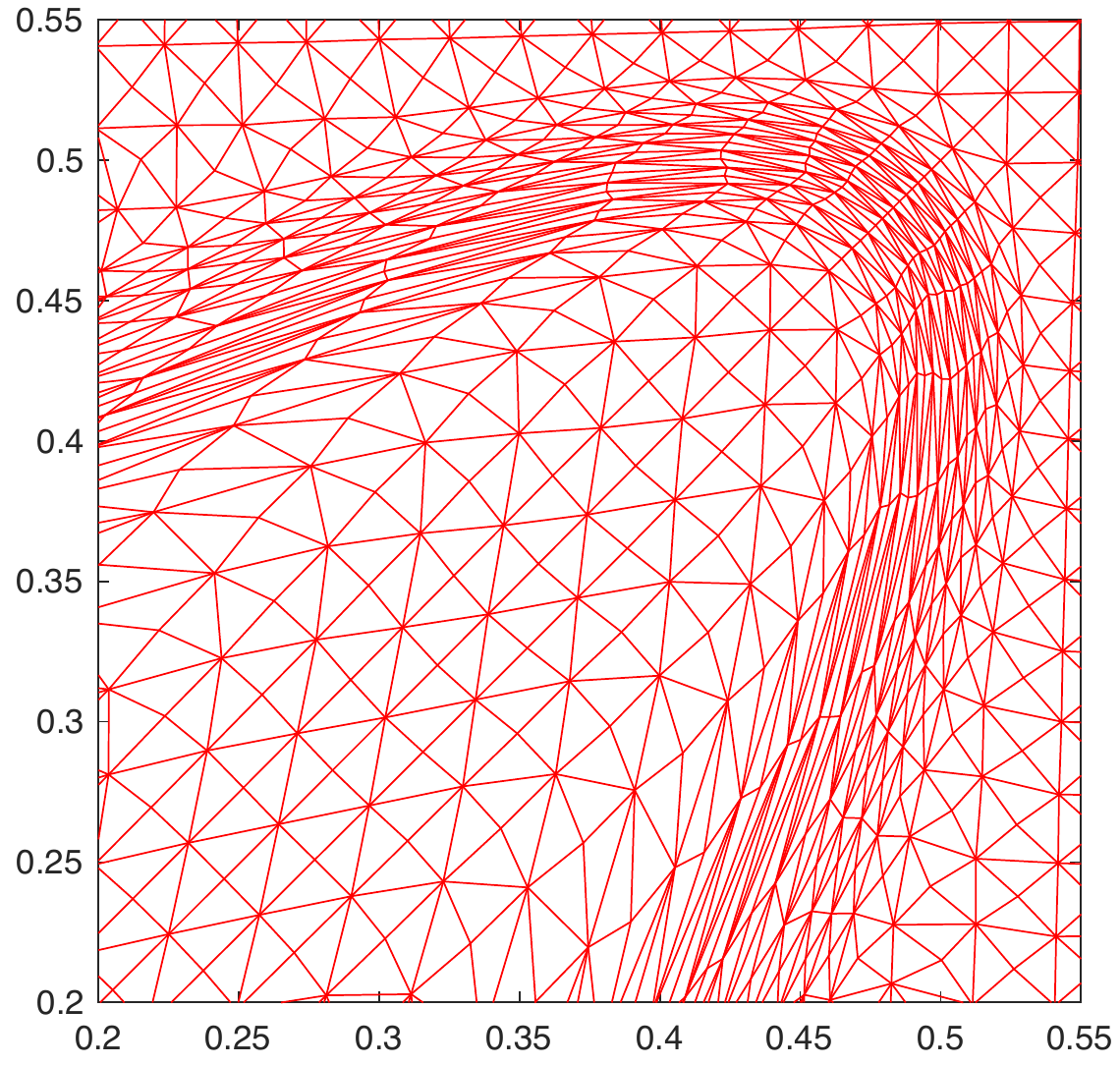}
\end{center}
\end{minipage}
}
\hbox{
\begin{minipage}[t]{2.1in}
\begin{center}
\includegraphics[width=2.1in]{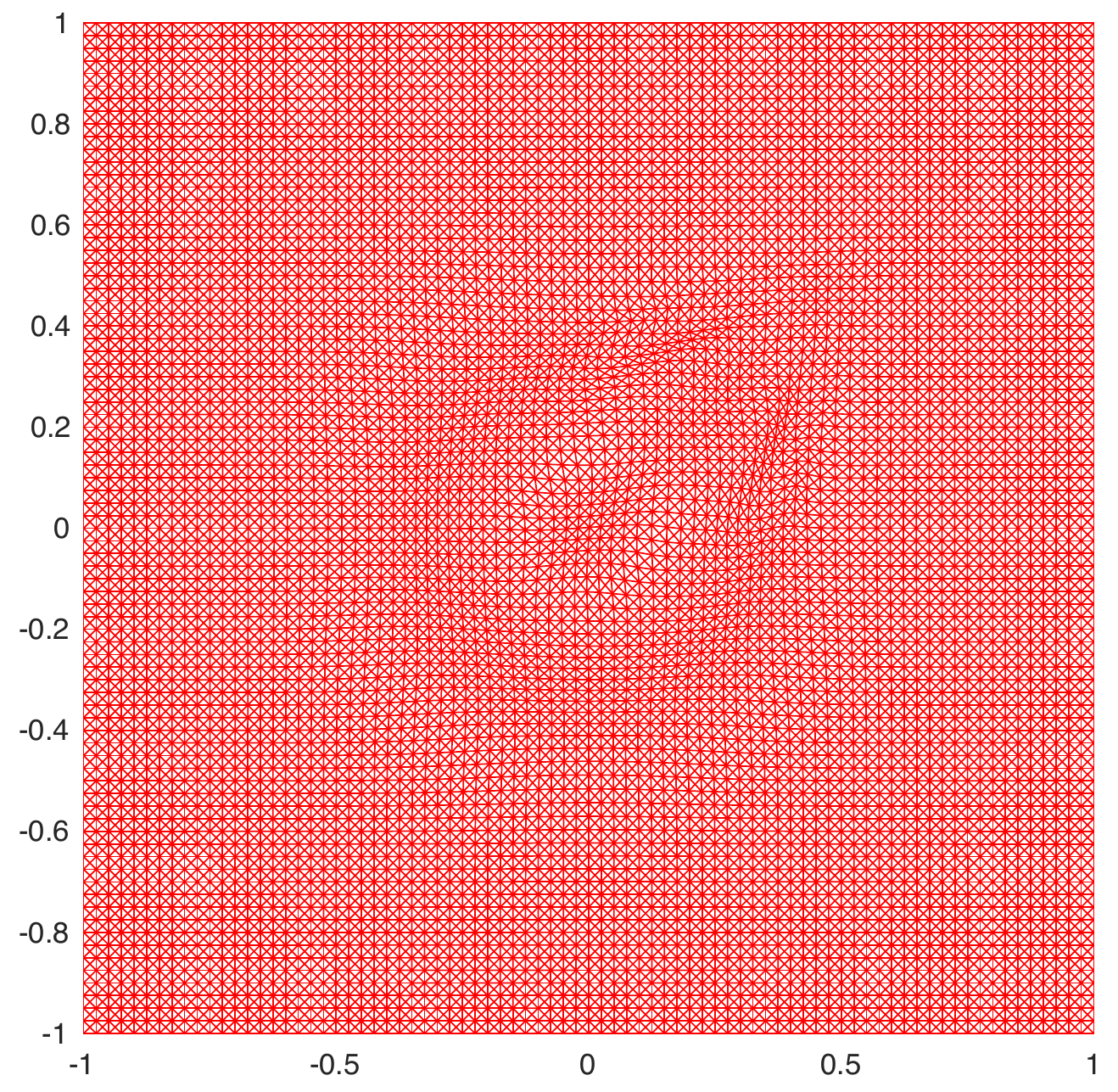}
\end{center}
\centerline{(b) Existing functional}\
\end{minipage}
\hspace{2mm}
\begin{minipage}[t]{2.1in}
\begin{center}
\includegraphics[width=2.1in]{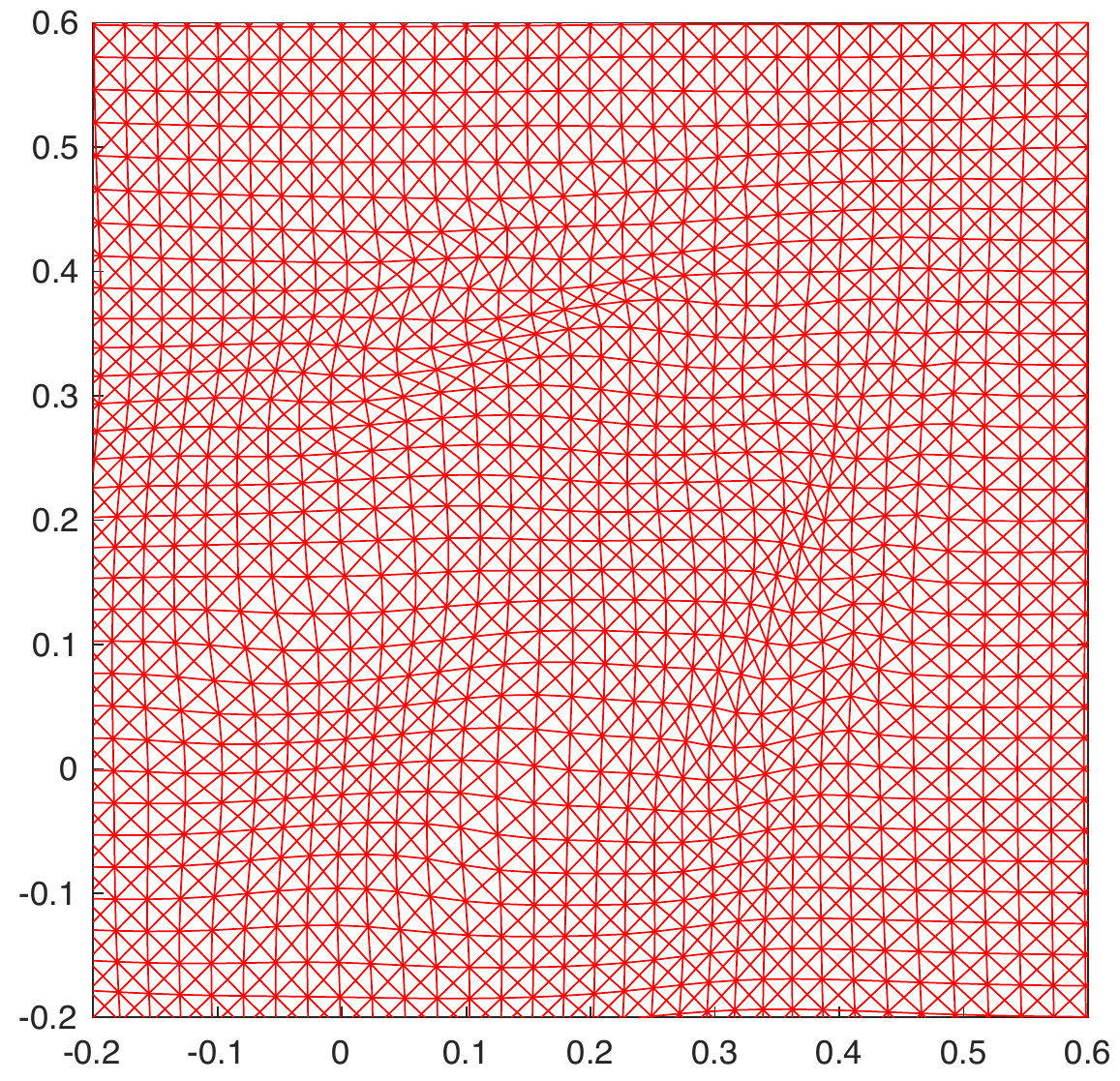}
\end{center}
\end{minipage}
\hspace{2mm}
\begin{minipage}[t]{2.1in}
\begin{center}
\includegraphics[width=2.1in]{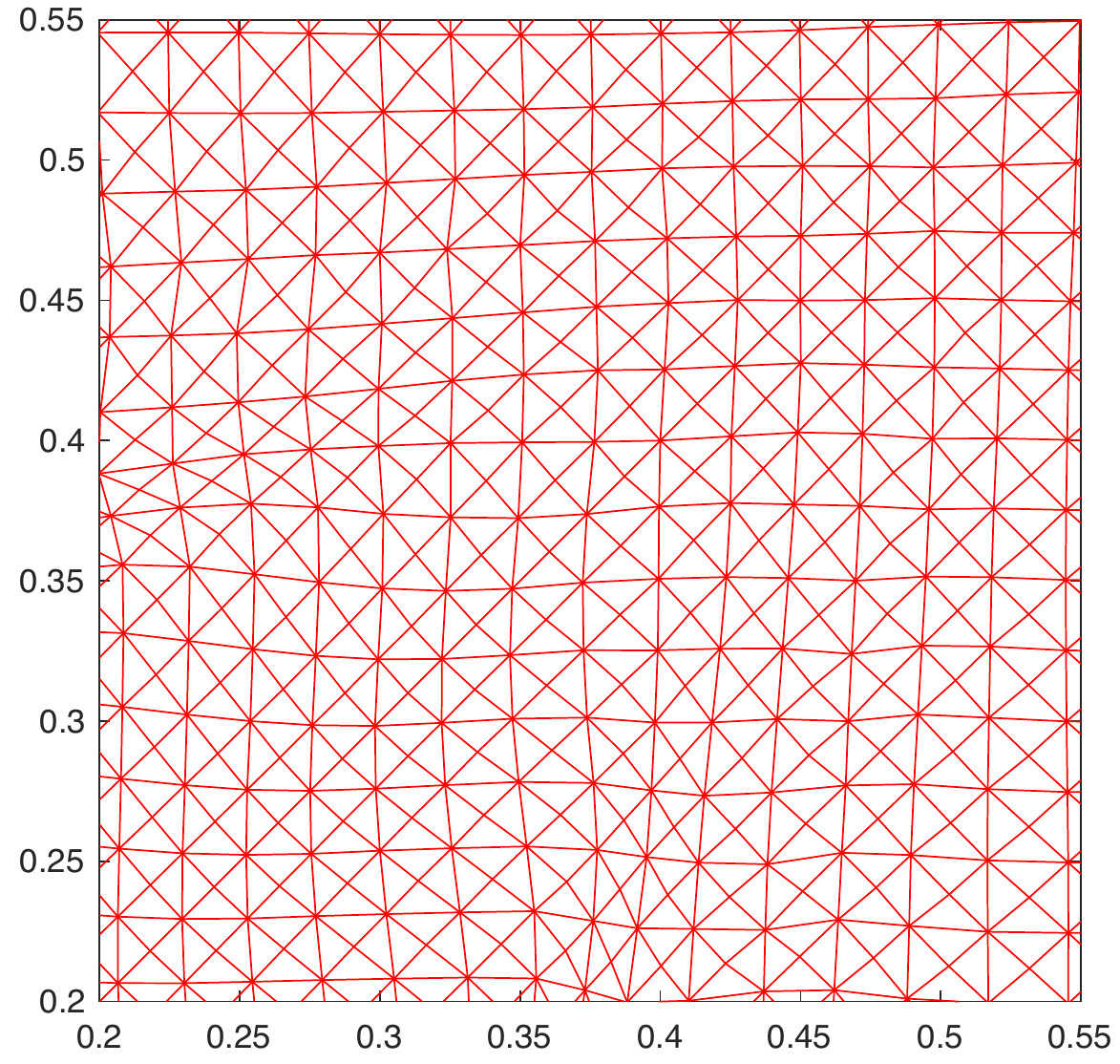}
\end{center}
\end{minipage}
}
\caption{Example~\ref{Exam5.3}. Example meshes (left), close-ups near the the tip (middle),
and a closer version of the tip (right) with $N=25600$.}
\label{fig:burgers}
\end{center}
\end{figure}

\begin{table}[htb]
\caption{Mesh quality measures for Example~\ref{Exam5.3}.}
 \begin{center}
\begin{tabular}{| c | c | c | c | c | c |} 
\hline
  Functional & N & $Q_{geo}$ & $Q_{eq}$ & $Q_{ali}$ \\
\hline
\multirow{3}{6em}{Existing} & 1600 & 1.502 & 5.696 & 1.842 \\ 
& 6400 & 1.934 & 14.20&  2.391\\
& 25600 &  1.677 & 31.77 & 3.426 \\ 
\hline
\multirow{3}{6em} {New} & 1600 & 2.130 & 4.705 & 1.577 \\ 
& 6400 &  8.215 & 6.470 & 2.731\\ 
& 25600 & 17.01 & 14.68 & 4.7111 \\ 
\hline
\end{tabular}
\label{table-burgers}
\end{center}
\end{table}
\end{exam}

\section{Conclusions and further comments}
\label{SEC:conclusion}

In the previous sections, we have introduced a new functional based on the equidistribution and alignment conditions.
The functional is formulated by directly combining these two conditions into one with only a single parameter. It should be pointed out that ($\ref{Huang2}$) does not contain $\theta$, a parameter that requires one to try to effectively balance the equidistribution and alignment conditions in ($\ref{Huang1}$).
We have proven a number of theoretical results for this new functional at the discrete level which are similar
to those of an existing functional that is also based on the equidistribution and alignment conditions
but contains an additional parameter. For example, the new functional
was proven to be coercive (Theorem~\ref{coercivity}).  With this, it was then shown that the element altitude and volumes of the mesh trajectory of the discrete MMPDE
associated with the new functional are bounded away from zero and the mesh trajectory stays nonsingular for all time
if it is nonsingular initially (Corollary 4.1). Moreover, Corollary 4.2 states that the value of the meshing functional decreases
monotonically along the mesh trajectory, while the latter has limit meshes that are critical points of the meshing functional.

The numerical results shown in this paper demonstrated that the new functional produces correct mesh concentration
and its performance is comparable to that of the existing functional which has been used successfully for various
applications. In addition, the numerical results validated the theoretical properties of the new functional.
It was shown that the meshing functional was monotonically decreasing and the minimum volume of the mesh element was bounded below as functions of time. From these results, we conclude that the new functional
is similar to the existing functional in both numerical performance and theoretical properties.

It should be noted that the numerical experiments provided in this work are limited.
In order to better understand the performance of the new functional, more work and
a variety of examples are necessary. Specifically, one of the main disadvantages of the new functional
is that it is not convex whereas the existing functional is known to be polyconvex and can be made
convex with the special choice of the parameter $\theta$ ($\theta = 1/2$).
With this in mind, it is hard to say how the non-convexity of the new functional affects the numerics. 
For the examples we tested, we did not experience any difficulty with computation or CPU time
but problems may occur in other examples.  This may be a topic for further investigations.\\

{\textbf{Acknowledgement.}} The authors would like to thank the anonymous referees for their valuable comments in improving the quality of the paper.


\end{document}